\documentclass[12pt]{amsart}
\usepackage{geometry} 

\usepackage[usenames]{color}
\usepackage{amssymb}
\usepackage{amsthm}
\usepackage{tikz}
\usetikzlibrary{matrix,arrows}

\def\rouge{ \textcolor{red} }

\def\unprotectedboldentry#1{\textcolor{Red}{\textbf{#1}}}
\def\boldentry{\protect\unprotectedboldentry}
\usepackage{tikz}
\usetikzlibrary{calc}
\usepackage{ifthen}
\newcommand{\tikztableau}[2][scale=0.6,every node/.style={font=\small}]{
    \def\newtableau{#2}
    \begin{array}{c}
    \begin{tikzpicture}[#1]
    \coordinate (x) at (-0.5,0.5);
    \coordinate (y) at (-0.5,0.5);
    \foreach \row in \newtableau {
        \coordinate (x) at ($(x)-(0,1)$);
        \coordinate (y) at (x);
        \foreach \entry in \row {
            \ifthenelse{\equal{\entry}{X}}
               {
                \node (y) at ($(y) + (1,0)$) {};
                \fill[color=gray!10] ($(y)-(0.5,0.5)$) rectangle +(1,1);
                \draw[color=gray] ($(y)-(0.5,0.5)$) rectangle +(1,1);
               }
               {
                \ifthenelse{\equal{\entry}{\boldentry X}}
                   {
                    \node (y) at ($(y) + (1,0)$) {};
                    \fill[color=gray] ($(y)-(0.5,0.5)$) rectangle +(1,1);
                    \draw ($(y)-(0.5,0.5)$) rectangle +(1,1);
                   }
                   {
                    \node (y) at ($(y) + (1,0)$) {\entry};
                    \draw ($(y)-(0.5,0.5)$) rectangle +(1,1);
                   }
               }
            }
        }
    \end{tikzpicture}
    \end{array}}
\newcommand{\tikztableausmall}[1]{\tikztableau[scale=0.45,every node/.style={font=\rm\small}]{#1}}
\newcommand{\tikztableautiny}[1]{\tikztableau[scale=0.25,every node/.style={font=\rm\tiny}]{#1}}

 \newtheorem{theorem}{Theorem}[section]
\newtheorem{lemma}[theorem]{Lemma}
\newtheorem{proposition}[theorem]{Proposition}
\newtheorem{corollary}[theorem]{Corollary}

\theoremstyle{definition}
\newtheorem{example}[theorem]{Example}
\newtheorem{remark}[theorem]{Remark}

\newtheorem{definition}[theorem]{Definition}

\numberwithin{equation}{section}

\def\u#1{{\bf u}_{#1}}
\def\a{a}
\def\b{b}
\def\c{c}
\def\d{d}
\def\S{{\mathcal S}}
\def\t#1{{\bf t}_{#1}}
\def\s#1{{\bf s}_{#1}}

\def\HH{{\bf h}}

\title [FG monoids and Pieri operations]{Fomin-Greene monoids and Pieri operations} 
\author{Carolina Benedetti and Nantel Bergeron}
\address[Nantel Bergeron]
 {Department of Mathematics and Statistics\\York University\\
 Toronto, Ontario M3J 1P3\\
 CANADA}
 \email{bergeron@mathstat.yorku.ca}
 \urladdr{http://www.math.yorku.ca/bergeron}

\address[Carolina Benedetti]
 {Department of Mathematics and Statistics\\York University\\
 Toronto, Ontario M3J 1P3\\
 CANADA}
 \email{carobene@mathstat.yorku.ca}

\thanks{N. Bergeron is supported in part by NSERC}

\begin{document}

\begin{abstract} We explore monoids generated by operators on certain infinite partial orders. Our starting point is the work of Fomin and Greene on monoids satisfying the relations  $(\u{r}+\u{r+1})\u{r+1}\u{r}=\u{r+1}\u{r}(\u{r}+\u{r+1})$
 and $\u{r}\u{t}=\u{s}\u{r}$ if $|r-t|>1.$
Given such a monoid, the non-commutative functions in the variables $\u{}$ are shown to commute. Symmetric functions in these operators often encode interesting structure constants. 
Our aim is to introduce similar results for more general monoids not satisfying the relations of Fomin and Greene. This paper is an extension of a talk by the second author at the workshop on algebraic monoids, group embeddings and algebraic combinatorics at The Fields Institute in 2012.
\end{abstract}

\maketitle

\section{Introduction}\label{intro}
In their work on the plactic monoid, Lascoux and Sch\i{u}tzenberger~\cite{LS81} constructed the Schur functions in terms of noncommutative variables satifying only Knuth relations.  It was subsequently discovered that symmetric functions can be constructed using different monoid algebras, for example the nil-plactic monoid, the nil-coxeter monoid or the $H_n(0)$ algebra. A uniform understanding of these constructions can be found in 
the seminal work of Fomin and Greene~\cite{FG98}.

One of the main advantages of the work of Fomin and Green is that it shows the Schur positivity of certain generating functions defined on those monoid algebras.
This is a central problem in algebraic combinatorics and we still have several open problems of this kind. 
The theory of~\cite{FG98} worked very well for the problems it is set to solve, but it also has its limitations.

Here we want to show that this quest of understanding symmetric functions inside a monoid algebra is  very alive and new results are still underway and needed.
In this presentation, very close to the approach of Fomin  and Greene, we look at monoids generated by operators acting on an infinite poset. 
We show that a certain space of functions on the monoid algebra of operators is isomorphic to symmetric functions (or a subspace of symmetric functions). 
These subspaces are obtained via Pieri operators as defined in~\cite{bmws}. The posets we consider are very often produced from a Combinatorial Hopf Algebra as defined in~\cite{ABS,BLL}. Unlike the theory in ~\cite{FG98}, we are not guaranteed to have Schur positivity. Even when the object in question is Schur positive the rule of Fomin and Greene is not applicable. One has to develop new techniques to deal with this. It has been done in some cases, but it is still open in others.

We keep this paper as a talk, like a story. We introduce the results as they come from the examples. In the first part, Section~\ref{sc:classic}, we look at a classical example. Next, in Section~\ref{sc:New}, we look at less known examples and constructions which are unrelated to~\cite{FG98}. We then look at what can be done in the future in Section~\ref{sc:future}.

\section{A classical example}\label{sc:classic}

\subsection{Operators on the Young lattice}
We start by a classical construction of Schur functions inspired by~\cite{F95}. A \emph{partition} of an integer $n$ is a sequence of integers $\lambda=(\lambda_1,\lambda_2,\ldots,\lambda_\ell)$ such that $n=\lambda_1+\lambda_2+\cdots+\lambda_\ell$ and $\lambda_1\ge\lambda_2\ge\cdots\ge\lambda_\ell>0$.
When $\lambda$ is a partition of $n$ we denote it by $\lambda\vdash n$. We also denote the number of parts of $\lambda$ by $\ell(\lambda)=\ell$ and its size by $|\lambda|=n$. The diagram of a partition $\lambda$, also denoted $\lambda$, is the subset of ${\mathbb Z}\times{\mathbb Z}$ given by $\lambda=\big\{ (i,j) : 1\le j\le \ell,\ 1\le i\le \lambda_j \big\}$. We draw this by putting a unit box with coordinates $(i,j)$ in the bottom left corner. For example the partition $\lambda=(4,2,1)$ is depicted by
$$ \tikztableautiny{{X},{X,X},{X,X,X,X}}.$$
The Young lattice $\mathcal Y$ consists of all partitions $\lambda\vdash n\ge 0$, ordered by inclusion of diagrams. The empty partition is the unique partition for $n=0$.
An inclusion $\mu\subset\lambda$ is a cover if and only if $\mu \cup \{(i,j)\}=\lambda$ for a unique cell $(i,j)$. We will label such a cover by an edge labeled by $c_{i,j}=j-i$:
 $$\mu  \ \stackrel{c_{i,j}}{\longrightarrow}\
              \lambda$$
We can draw the lower part of this poset as 
$$
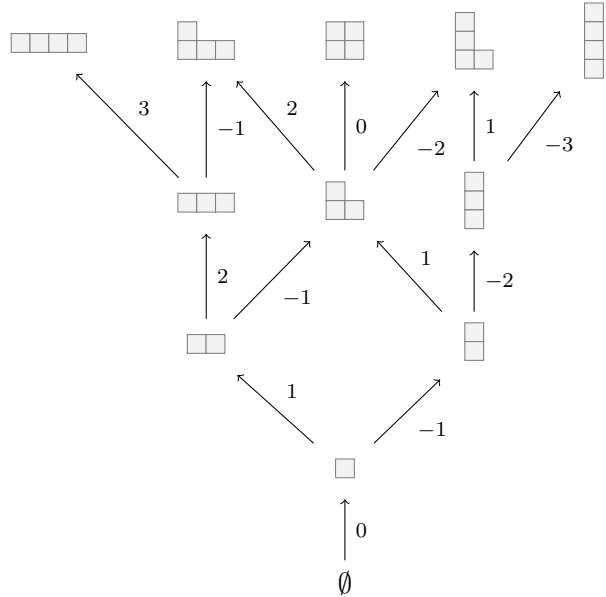
\begin{tikzpicture}[description/.style={fill=white,inner sep=2pt, font=\tiny }]
\matrix (m) [matrix of math nodes, row sep=2em,
column sep=1.5em, font=\small]
{ \tikztableautiny{{X,X,X, X}} & \tikztableautiny{{X},{X,X, X}}  & \tikztableautiny{{X,X},{X, X}}  & \tikztableautiny{{X},{X}, {X,X}} & \tikztableautiny{{X},{X},{X},{X}} \\
&  \tikztableautiny{{X,X,X}} & \tikztableautiny{{X},{X,X}}  & \tikztableautiny{{X},{X},{X}} & \\
&  \tikztableautiny{{X,X}} &   & \tikztableautiny{{X},{X}} & \\
&  &  \tikztableautiny{{X}} &   & \\
&  &  \emptyset&   & \\
 };
\path[->] (m-5-3) edge node[auto,swap, font=\tiny] {$  0 $} (m-4-3);
\path[->] (m-4-3) edge node[auto,swap, font=\tiny] {$  1 $} (m-3-2);
\path[->] (m-4-3) edge node[auto,swap, font=\tiny] {$  -1 $} (m-3-4);
\path[->] (m-3-2) edge node[auto,swap, font=\tiny] {$  2 $} (m-2-2);
\path[->] (m-3-2) edge node[auto,swap, font=\tiny] {$  -1 $} (m-2-3);
\path[->] (m-3-4) edge node[auto,swap, font=\tiny] {$  1 $} (m-2-3);
\path[->] (m-3-4) edge node[auto,swap, font=\tiny] {$  -2 $} (m-2-4);
\path[->] (m-2-2) edge node[auto,swap, font=\tiny] {$  3 $} (m-1-1);
\path[->] (m-2-2) edge node[auto,swap, font=\tiny] {$  -1 $} (m-1-2);
\path[->] (m-2-3) edge node[auto,swap, font=\tiny] {$  2 $} (m-1-2);
\path[->] (m-2-3) edge node[auto,swap, font=\tiny] {$  0 $} (m-1-3);
\path[->] (m-2-3) edge node[auto,swap, font=\tiny] {$  -2 $} (m-1-4);
\path[->] (m-2-4) edge node[auto,swap, font=\tiny] {$  1 $} (m-1-4);
\path[->] (m-2-4) edge node[auto,swap, font=\tiny] {$  -3 $} (m-1-5);
\end{tikzpicture}
$$

Consider the free ${\mathbb Z}$-module ${\mathbb Z}{\mathcal Y}$ spanned by all partitions of $n\ge 0$.
We  define linear operators  $\u{r}$ for each $r \in {\mathbb Z}$ as follows
\begin{equation}\label{eq:op_partition}
\begin{array}{rcl}
\u{r} \colon\  {\mathbb Z}{\mathcal Y}&\longrightarrow& \quad{\mathbb Z}{\mathcal Y},\\
\mu\quad&\longmapsto&\ \ \rule{0pt}{28pt} \left\{\begin{array}{ll}
\lambda 
&\mbox{ if }  \mu  \ \stackrel{r}{\longrightarrow}\
              \lambda \mbox{ in } {\mathcal Y}
\medskip       
\\      0& \mbox{ otherwise.}\end{array}\right.
\end{array} 
\end{equation}
For example 
$$ \u{0}\big( \tikztableautiny{{X},{X, X}} \big) =\tikztableautiny{{X,X},{X, X}} \qquad \mbox{ and } \qquad   \u{1}\big( \tikztableautiny{{X},{X, X}} \big) = 0. $$ 
We are interested in the monoid ${\mathcal M}\langle\u{r}\rangle$ generated by the operators $\u{r}$ for $r\in {\mathbb Z}$ and the zero operator $\bf 0$. By the nature of these operators,
it is not very hard to see that they satisfy the following relations:
\begin{equation} \label{eq:relop_partition}
\begin{array}{rcrclcll}
(1)&&\u{r}^2&=&{\bf 0}&&\\
(2)&&\u{r}\u{r+1}\u{r} = \u{r+1}\u{r}\u{r+1}&=&{\bf 0}&&\\
(3)&&\u{r}\u{t}&=&\u{s}\u{r}&& \hbox{if $|r-t|>1$}.\hfill\\
\end{array}  
 \end{equation}
 These relations can be understood graphically. The first relation states that once we add a cell in a given diagonal, if we try to add a second cell in the same diagonal we will not get a partition:
$$ 
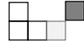
\begin{tikzpicture} 
 \node (1) at (0,0) {$ \tikztableautiny{{ },{ , , X}}$}; 
 \node (2) at (.5,.17) {$ \tikztableautiny{{\boldentry  X}}$}; 
 \end{tikzpicture}  \, . $$
  The second relation states that if we add two consecutive cells in a row (or column) and if we try to add a third cell in the same diagonal as the first added cell we will not get a partition:
$$ \begin{tikzpicture} 
 \node (1) at (0,0) {$ \tikztableautiny{{ },{ ,X, X}}$}; 
 \node (2) at (.25,.17) {$ \tikztableautiny{{\boldentry X}}$}; 
 \end{tikzpicture}  \, .$$
 The third relation states that we can add  two cells independently in diagonals that are far from each other:
$$ \tikztableautiny{{\boldentry X},{ },{ , , ,\boldentry X}} \, .
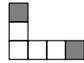
$$

\begin{proposition}\label{prop:relY}
 ${\mathcal M}\langle\u{r}\rangle$ is the monoid freely generated by the $\u{r}$ for $r\in{\mathbb Z}$ and $\bf 0$
 modulo the relations~\eqref{eq:relop_partition}.
\end{proposition}

This is a consequence of a more general theorem and it can be shown using some very well known facts about the symmetric group and the combinatorics of partitions.
However to our knowledge this statement is not mentioned as such in the literature.
To see that the relations~\eqref{eq:relop_partition} generate all the relations of the monoid ${\mathcal M}\langle\u{r}\rangle$ requires a deeper understanding of the relations. We will sketch a proof here. Recall that the symmetric group is generated by simple reflections $s_r$ satisfying the braid relations:
\begin{equation} \label{eq:relsym_group}
\begin{array}{rcrclcll}
(1)&&s_r^2&=&Id&&\\
(2)&&s_rs_{r+1}s_r&=&s_{r+1}s_rs_{r+1}&&\\
(3)&&s_r s_t&=&s_ts_r&& \hbox{if $|r-t|>1$}.\hfill\\
\end{array}  
 \end{equation}
For a permutation $w$, the length $\ell(w)$ is the minimal number of generators $s_r$ necessary to express $w$ as a product of generators.
If $w=s_{i_1}s_{i_2}\cdots s_{\ell(w)}$, then we say that the word $s_{i_1}s_{i_2}\cdots s_{\ell(w)}$ is a reduced word for $w$. There is a small abuse of notation here: a reduced word is an element of the free monoid generated by the $s_r$'s. Here, we are studying  the equivalence classes of words modulo the relations~\eqref{eq:relsym_group}.
It is a well known fact that any two reduced words for a given permutation $w$ are connected together using only (2) and (3)  of the relations~\eqref{eq:relsym_group}. 
Moreover, if a word $s_{i_1}s_{i_2}\cdots s_{k}$ is not reduced, then at least one instance of the relation (1) of~\eqref{eq:relsym_group} will be used to reduce it  (see \cite{humphreys}). The set of equivalence classes of words that do not have any occurrence of $s_rs_{r+1}s_r$ are in bijection with 321-avoiding permutations. These are permutations $w$ with no $i<j<k$ such that $w(i)>w(j)>w(k)$ (see \cite{stan84}). 

Consider now the infinite group $\S_{\mathbb Z}$  of permutations of $\mathbb Z$ with only finitely many non-fixed points. This is the group generated by the simple reflections $s_r$ for $r\in{\mathbb Z}$ subject to the relations in ~\eqref{eq:relsym_group}. For $w\in \S_{\mathbb Z}$ we define the operator 
  $$ \u{w} = \u{i_1} \u{i_2}\cdots \u{i_{\ell(w)}} $$
where $s_{i_1}s_{i_2}\cdots s_{\ell(w)}$ is any reduced word for $w$. Comparing the relations~\eqref{eq:relsym_group} with the relations~\eqref{eq:relop_partition} we see that this is a well defined operator. Moreover, if $w$ is not 321-avoiding, then relation (2) of~\eqref{eq:relop_partition} gives $\u{w}=0$ and if $s_{i_1}s_{i_2}\cdots s_{k}$ is not a reduced word, then $\u{i_1} \u{i_2}\cdots \u{i_{k}} =0$. In order to show that the relations~\eqref{eq:relop_partition} generate all the relations of ${\mathcal M}\langle\u{r}\rangle$
it is enough to prove that
\begin{lemma}\label{prop:opw}\ 

\begin{enumerate}
\item[(a)] For each $w\in \S_{\mathbb Z}$ 321-avoiding, we have $\u{w}\ne 0$,
\item[(b)] For  $w,w'\in \S_{\mathbb Z}$ 321-avoiding, we have that $w\ne w'$ implies $\u{w}\ne \u{w'}$.
\end{enumerate}
\end{lemma}
This will indeed show that the map from the free monoid generated by the $\u{r}$'s modulo the relations~\eqref{eq:relop_partition}  to ${\mathcal M}\langle\u{r}\rangle$ has no kernel and is surjective. 
These results are known in some different form (see \cite{BJN,Stemb}) and are not trivial. We will provide a proof here in this context  for completeness.

Let us start with the lattice ${\mathcal Y}$ and its labelled covers. It is possible to encode this lattice and its covers with a subset of the 321-avoiding permumation in $ \S_{\mathbb Z}$. Given a partition $\lambda$, add the two positive $x$-$y$ axis. We put the numbers $\ldots, -3,-2,-1,0$ for every vertical step from infinity on the $y$-axis following the border of the partition. We put the numbers $1,2,3,\ldots$ one on each horizontal step from left to right. The example below describes this procedure better for $\lambda=(3,1)$,
$$
 \raise -0pt
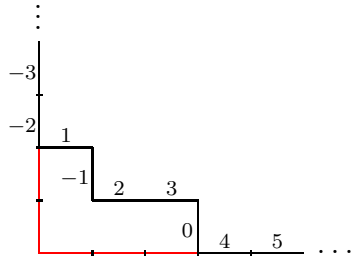
\hbox{ \begin{picture}(100,100)
 \rouge{   \put(0,0) {\line(1,0){60}}
               \put(0,0) {\line(0,1){40}} }
 \put(-1,20) {\line(1,0){2}}    \put(-1,40) {\line(1,0){2}}     \put(-1,60) {\line(1,0){2}}   
 \put(20,-1) {\line(0,1){2}}    \put(40,-1) {\line(0,1){2}}     \put(60,-1) {\line(0,1){2}}    \put(80,-1) {\line(0,1){2}}
 \put(0,40) {\line(0,1){40}}   \put(0,40) {\line(1,0){20}}  \put(20,40) {\line(0,-1){20}} \put(20,20) {\line(1,0){40}}
  \put(60,20) {\line(0,-1){20}} \put(60,00) {\line(1,0){40}}
  \put(-2,85){$\vdots$}    \put(105,0){$\ldots$}
  \put(-12,66){$\scriptstyle {-3}$}  
  \put(8,42){$\scriptstyle {1}$}  
  \put(-12,46){$\scriptstyle {-2}$}    \put(8,26){$\scriptstyle {-1}$}  
  \put(28,22){$\scriptstyle {2}$}    \put(48,22){$\scriptstyle {3}$}    \put(68,2){$\scriptstyle {4}$}  
  \put(54,6){$\scriptstyle {0}$}  
  \put(88,2){$\scriptstyle {5}$}     
      \end{picture}} 
$$
When we read the entries on the $y$-axis, then the outer boundary of $\lambda$ followed by the $x$-axis, we obtain a 321-avoiding permutation $v(\lambda)\in\S_{\mathbb Z}$ (the entries on the axis are fixed points). In the example above we get
$$v(\lambda)=( \cdots, -3,-2,1,-1,2,3,0,4,5,\cdots).$$
If we have a cover $\mu  \stackrel{r}{\longrightarrow}\lambda$, then the entry $v(\mu)(r)\le 0<v(\mu)(r+1)$.
Adding a box on the diagonal of content $r$ has the effect of interchanging these two entries in $v(\mu)$. We have shown the following:
\begin{lemma}\label{lem:op_mult}
 $$\mu  \stackrel{r}{\longrightarrow}\lambda \qquad \implies \qquad  v(\lambda)=v(\mu)s_r \quad \mbox{ and }\quad \ell(v(\lambda))=\ell(v(\mu))+1.$$
\end{lemma}

This lemma allows us to show Lemma~\ref{prop:opw} (b) if we know that $\u{w}\ne 0$. Indeed, if $\u{w}(\mu)=\lambda$, then the above lemma gives us that $v(\lambda)=v(\mu)w$. Hence if $w\ne w'$, then $v(\mu)w\ne v(\mu)w'$ and $\u{w}\ne \u{w'}$. 

Now, in order to prove  Lemma~\ref{prop:opw} (a)  we need to construct a partition $\mu$ such that
$\u{w}(\mu)=\lambda\ne 0$ for each 321-avoiding $w\in \S_{\mathbb Z}$.  When $\mu\subseteq\lambda$, we say that
the diagram $\lambda/\mu$ obtained by removing the cells of $\mu$ from $\lambda$ is a skew diagram.
For  $w\in \S_{\mathbb Z}$ that is 321-avoiding, we construct recursively on the length $\ell(w)$ a skew diagram $\lambda/\mu$ such that $\u{w}(\mu)=\lambda$. Moreover, if we read the content of the cells of $\lambda/\mu$, row by row, from left to right, starting at the bottom, then we get a sequence of integers $(j_1,j_2,\ldots,j_k)$ such that $s_{j_1}s_{j_2}\cdots s_{j_k}$ is a reduced word for $w$.  Finally,  if $(i,j)\in\mu$ and $(i+1,j)\not\in\mu$ and $(i,j+1)\not\in\mu$, then either $(i+1,j)\in\lambda$ or $(i,j+1)\in\lambda$ (see Example \ref{example2.4} below).

If $\ell(w)=0$, then the result is immediate as $\lambda/\mu=\emptyset/\emptyset$ does the trick. We assume that for all 321-avoiding permutations such that $\ell(w)<\ell$ we can construct $\lambda/\mu$ as above.
Let $w=s_{i_1}s_{i_2}\cdots s_{i_{\ell}}$ be a reduced expression for a 321-avoiding permutation of length $\ell(w)=\ell$. 
By induction hypothesis we assume we have constructed $\lambda/\mu$ for $w'=s_{i_1}s_{i_2}\cdots s_{i_{\ell-1}}$.
We can moreover assume that $(i_1,i_2,\ldots,i_{\ell-1})$ is the sequence of contents we read from $\lambda/\mu$.
We consider a cell on the diagonal of content $d=i_\ell$ sliding from infinity downward and stop at $(i,j)=(i,i+d)$ the first contact of either $\mu$, $\lambda/\mu$ or one of the $x$-$y$-axes. We claim that
      $$\mbox{if $(i-1,j-1)\in\lambda/\mu$, then both $(i-1,j)\in\lambda/\mu$ and $(i,j-1)\in\lambda/\mu$.}$$
In the sequence $(i_1,i_2,\ldots,i_{\ell-1})$, let $k$ be such that $(i_k,i_{k+1},\ldots,i_\ell)$ are the contents of the cells
in rows $i+1$ and up in $\lambda/\mu$. Since no cell of $\lambda/\mu$ is in column $j$ and up in row $i$ and up, we have that $i_{k'}<j-i-1=d-1$ for all $k\le k'\le\ell-1$. This means that $s_{i_{k'}}$ and $s_d$ commute for all $k\le k'\le\ell-1$. We have that 
\begin{equation}\label{eq:read}
s_{i_1}s_{i_2}\cdots s_{i_{\ell}}=s_{i_1}s_{i_2}\cdots s_d s_{i_k}\cdots s_{i_{\ell-1}}.
\end{equation}
Now suppose $(i,j-1)\not\in\lambda/\mu$. This means that $s_d$ commutes with all $s_{c}$ where $c$ is the content of cells in row $i$
of $\lambda/\mu$ and all cells of content $e$ in row $i-1$ and column $j'>j+1$. We depict this as follows
$$
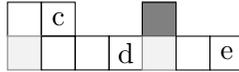
\begin{tikzpicture} 
 \node (1) at (0,0) {$ \tikztableausmall{{,c },{X,,, d}}$}; 
 \node (2) at (1.56,0.02) {$ \tikztableausmall{{\boldentry X},{ X, ,e}}$}; 
 \end{tikzpicture} 
$$
where the dark cell corresponds to the added cell in position $(i,j)$ of content $d$. Since the cell $(i,j-1)\not\in\lambda/\mu$ all cells in row $i$ have content $c<d-1$. The cells in row $i-1$ and column $j'>j$ have content $e>d+1$.
If $(i,j-1)\not\in\lambda/\mu$, then we \ get $s_ds_d$ in the reduced expression of $w$, a contradiction. If in addition $(i-1,j)\in\lambda/\mu$, then we get $s_ds_{d+1}s_d$ which contradicts  the fact that $w$ is 321-avoiding. Hence we must have that $(i,j-1)\in\lambda/\mu$. Now if we assume that $(i-1,j)\not\in\lambda/\mu$ and $(i,j-1)\in\lambda/\mu$ the picture is now
$$ \tikztableausmall{{,c,,, \boldentry X},{X,,, d}}\ .$$
The cell in position $(i-1,j-1)$ has content $d$. All cells in row $i$ and column $j'<j-1$ have content $c<d-1$. This time we can move the reflection $s_d$ corresponding to the cell in position $(i-1,j-1)$ to pass the $s_c$ in row $i$ up to $s_{d-1}s_d$. Again we get a contradiction as $s_ds_{d-1}s_d$ cannot occur in the reduced word of a 321-avoiding permutation $w$. This concludes the case when $(i-1,j-1)\in\lambda/\mu$. In this case we simply add the cell $(i,j)$ to $\lambda$ and not to $\mu$. The diagram $(\lambda\cup(i,j))/\mu$ is a skew shape with all the desired properties and the right hand side of~\eqref{eq:read} is the reduced word of $w$ that we read from this diagram.

We now consider the case where $(i-1,j-1)\in\mu$ or falls outside the first quadrant. If both $(i-1,j)\in\lambda/\mu$ and $(i,j-1)\in\lambda/\mu$, then again the diagram $(\lambda\cup(i,j))/\mu$ is a skew shape with all the desired properties and the right hand side of~\eqref{eq:read} is the reduced word of $w$ that we read from this diagram. By induction hypothesis, it is not the case that both $(i-1,j)\not\in\lambda/\mu$ and $(i,j-1)\not\in\lambda/\mu$. If $(i-1,j)\in\lambda/\mu$ and $(i,j-1)\not\in\lambda/\mu$, then we move all
the boxes of $\lambda/\mu$ in row $r\ge i$ up each diagonal by 1 unit. This increases the size of $\lambda$ and $\mu$ proportionally but keeps the relative shape of $\lambda/\mu$ invariant along the diagonal lines. We then add the box $(i,j)$ to lambda and add all the boxes $(i',j)$ for $i'<i$ to both $\lambda$ and $\mu$. Graphically we have
$$
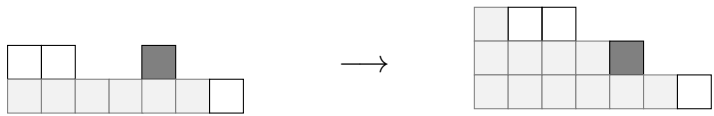
\lower24pt\hbox{$\begin{tikzpicture} 
 \node (1) at (0,0) {$ \tikztableausmall{{, },{X,X,X,X}}$}; 
 \node (2) at (1.56,0) {$ \tikztableausmall{{\boldentry X},{ X,X, }}$}; 
 \end{tikzpicture} $}
 \qquad\longrightarrow\qquad
 \tikztableausmall{{X,,},{X,X,X,X,\boldentry X},{X,X,X,X,X,X, }}
$$
The case where $(i-1,j)\not\in\lambda/\mu$ and $(i,j-1)\in\lambda/\mu$ is exactly transposed, interchanging the roles of row and column. In any case we obtain the desired skew shape $\lambda'/\mu'$ such that $\u{w}(\mu')=\lambda'\ne 0$.

\begin{example}\label{example2.4}
Let us illustrate the induction procedure involved in the proof of Lemma \ref{prop:opw} $(a)$. Start with $w=s_{3}s_{-3}s_4s_2$ and its skew shape as illustrated on the left hand side of the figure below. The induction step tells us that the operator $\u{ws_0}$ is not zero since $\u{ws_0}(\mu)=\lambda$ where $\mu=(6,4,4,3,1)$ and $\lambda = (6,6,5,4,2)$:
 $$  \tikztableausmall{{-3},{X,X,X,X,2},{X,X,X,X,3,4 },{X,X,X,X,X,X}}  \qquad\longrightarrow\qquad
    \tikztableausmall{{X,-3},{X,X,X,0},{X,X,X,X,2},{X,X,X,X,3,4 },{X,X,X,X,X,X}}  
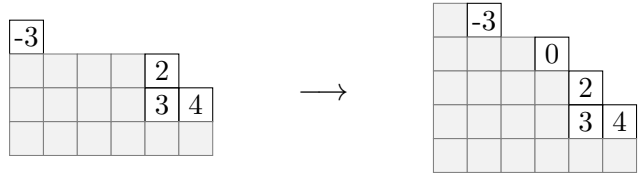
 $$
\end{example}

\subsection{Pieri operators on Young lattice and symmetric functions}\label{subsection:ClassicPieri}

In the previous section, we obtained a very good understanding of the noncommutative monoid ${\mathcal M}\langle\u{r}\rangle$.
We now introduce a commutative algebra ${\bf B}\langle H_{k}\rangle$ that is isomorphic to the (Hopf) algebra of symmetric functions $Sym$.
The algebra ${\bf B}\langle H_{k}\rangle$ is generated by certain homogeneous series $H_k$ in the elements of ${\mathcal M}\langle\u{r}\rangle$.
This is using the Pieri operators theory as developed in~\cite{bmws} related to the multiplication of symmetric functions (see~\cite{macdonald}).  

There are several combinatorial Hopf algebras of interest for our study. As it turns out, $\mathcal Y$ is  intimately related to $Sym$.
The space of symmetric functions is well known to have different bases indexed by partitions. We refer the reader to~\cite{macdonald,sagan} for
more details about our presentation of $Sym$.
We use the standard notation for the common bases of $Sym$:
$h_\lambda$ for complete homogeneous;
$e_\lambda$ for elementary;
$m_\lambda$ for monomial;
and $s_\lambda$ for Schur functions.
For simplicity, we let $h_i$ and $e_i$ denote the
corresponding generators indexed by the partition $(i)$.

There is a correspondence between the representation theory of all symmetric groups and symmetric functions.
The multiplication and comultiplication in $Sym$ corresponds to some induction and restriction of representations.
In this identification, Schur functions encode irreducible representations. In particular we must have that the coefficients
$C_{\lambda,\mu}^\nu$ in 
\begin{equation}\label{eq:LR}
 s_\lambda s_\mu =\sum_v C_{\lambda,\mu}^\nu s_\nu 
 \end{equation}
  are non-negative integers. They count the multiplicity of an irreducible in certain induced representations.
  This shows the nonnegativity of the constants $C_{\lambda,\mu}^\nu $ but does not give us a combinatorial formula for them.
  One is interested in a positive combinatorial rule to describe these numbers. This combinatorial rule is classically known as the Littlewood-Richardson rule.
  A particular case of this rule is Pieri rule that describes the multiplication by $h_k$:
    $$s_\lambda h_k =\sum_{\nu/\lambda \mbox{ a $k$-row strip}} s_\nu $$
where a $k$-row strip is a diagram with $k$ cells in distinct columns. In terms of the lattice $\mathcal Y$ we have the following characterization of $k$-row strip.

\begin{lemma}\label{lem:path}
 $\nu/\lambda$ is a $k$-row strip if and only if there is a strictly increasing path of length $k$ in $\mathcal Y$ from $\lambda$ to $\nu$. Moreover, if such a path exists from $\lambda$ to $\nu$, then it is unique. 
\end{lemma}

\proof  If $\nu/\lambda$ is a $k$-row strip, then we can add the cells of $\nu/\lambda$ to $\lambda$ one by one from left to right. Since the cells are in distinct columns, they are in distinct diagonals as well. Adding them from left to right will give us the desired strictly increasing path from  $\lambda$ to $\nu$. Conversely,  if we have a strictly increasing path from $\lambda$ to $\nu$, then the cells of  $\nu/\lambda$ are in distinct diagonals. Assume two cells of $\nu/\lambda$ are in the same column as pictured bellow
$$\tikztableausmall{{X,X,X,X,A},{X,X,X,X, },{X,X,X,X,B, }} $$
The cell $A$ has content strictly smaller than the content of the cell $B$. In the path from $\lambda$ to $\nu$ the cell $A$ would be added before the cell $B$.
But this is a contradiction since when the cell $A$ is added without cell $B$ this would not be a partition. Hence, $\nu/\lambda$ is a $k$-row strip.
\qed

This allows us to reconstruct the multiplication by $h_k$ using operators on $\mathcal Y$. Let 
  $$H_k= \sum_{i_1<i_2<\ldots<i_k} \u{i_k}\cdots\u{i_2}\u{i_1}.$$
This is an infinite series of operators of degree $k$ in   ${\mathcal M}\langle\u{r}\rangle$. In view of Lemma~\ref{prop:opw}, no term in the series $H_k$ vanishes.
If one fixes $\lambda$, there are only finitely many paths of length $k$ from $\lambda$ in $\mathcal Y$. This means that $H_k\colon {\mathbb Z}{\mathcal Y}\to {\mathbb Z}{\mathcal Y}$ is a well defined operator. 

\begin{proposition}
   $$H_k = \sum_{\ell(\zeta)=k} \u{\zeta},$$
   where $\zeta$ runs over all permutations such that its disjoint cycle decomposition $\zeta=C_1C_2\cdots C_s$ has only cycles of the form $C_i=(a+b,  \ldots,a+1,a)$ for some $a,b\in{\mathbb Z}$
   and $b>0$.
\end{proposition}

\begin{proof} It suffices to show that
  $$\u{i_k}\cdots\u{i_2}\u{i_1}  = \u{\zeta}$$
with $i_1<i_2<\ldots<i_k$ if and only if $\zeta$ decomposes into disjoint cycles of the form  $(a+b,  \ldots,a+1,a)$. The disjoint cycles  of $\zeta= s_{i_k}\cdots s_{i_2}s_{i_1} $
for $i_1<i_2<\ldots<i_k$ correspond to the consecutive segments $s_{a+b}\cdots s_{a+1} s_a = (a+b,  \ldots,a+1,a)$.
\end{proof}

Using Lemma~\ref{lem:path}
  $$ H_k(\lambda)=\sum_{\nu/\lambda \mbox{ a $k$-row strip}} \nu  \qquad \qquad \iff \qquad \qquad s_\lambda h_k =\sum_{\nu/\lambda \mbox{ a $k$-row strip}} s_\nu\,.$$
This implies that
  $$ H_b H_a(\lambda) = \sum_{\nu} d_{\lambda,(a,b)}^\nu \nu \qquad \qquad \iff \qquad \qquad s_\lambda h_a h_b =\sum_{\nu} d_{\lambda,(a,b)}^\nu s_\nu\,.$$
In particular, for all $\lambda$ we have  $H_b H_a(\lambda)=H_a H_b(\lambda)$ since $h_ah_b=h_bh_a$. Again the result below is derived from very classical results.

\begin{theorem}\label{thm:youngop}
 The algebra  ${\bf B}\langle H_{k}\rangle$ spanned by $\{H_1, H_2, H_3,\ldots\}$ is isomorphic to $Sym$.
\end{theorem}

\proof
We have seen above that $H_b H_a(\lambda)=H_a H_b(\lambda)$, but to see that the product of series $H_b H_a=H_a H_b$ requires a little bit more argument.
As we multiply $H_aH_b$ and $H_bH_a$, some terms will go to zero and others will survive. The terms that survive in  $H_aH_b$  are  of the form 
 $$ \u{w} = \u{i_1}\u{i_2}\cdots\u{i_a}\u{j_1}\u{j_2}\cdots\u{j_b}$$
 where $w$ is 321-avoiding, $i_1<i_2<\cdots<i_a$ and $j_1<j_2<\cdots<j_b$. Showing that $H_aH_b=H_bH_a$ requires the construction of a bijection between the possible reduced 
 expressions of $w=s_{i_1}\cdots s_{i_a}s_{j_1}\cdots s_{j_b}$ and the ones of the form $w=s_{j'_1}\cdots s_{j'_b}s_{i'_1}\cdots s_{i'_b}$ where $i'_1<i'_2<\cdots<i'_a$ and $j'_1<j'_2<\cdots<j'_b$.  This is done in~\cite{stan84} and also in~\cite{bsottileskew} using jeu-de-taquin. We then have that $\big\{H_\mu=H_{\mu_1}H_{\mu_2}\cdots H_{\mu_1}: \mu \mbox{ partition}\big\}$ spans ${\bf B}\langle H_{k}\rangle$.
To see that the $H_\mu$ are linearly independent, it suffices to remark that
  $$H_\mu(\emptyset)=\mu$$
  hence they have distinct values on $\emptyset$.
\qed

\begin{remark} Theorem~\ref{thm:youngop} follows easily from the more general Theorem~1.1 of~\cite{FG98}. The approach of Fomin and Greene has the advantage 
that one does not need to fully have all the relations of the $\u{r}$. It is enough to show that they satisfy the relations:
\begin{equation} \label{eq:relFG}
\begin{array}{rcrclcll}
(1)&&\u{r}\u{t}&=&\u{s}\u{r}&& \hbox{if $|r-t|>1$},\hfill\\
(2)&&(\u{r}+\u{r+1})\u{r+1}\u{r}&=&\u{r+1}\u{r}(\u{r}+\u{r+1})&& \hfill\\
\end{array}  
 \end{equation}
 It is clear that our operators $\u{r}$ satisfy the relations~\eqref{eq:relFG}.
 In later sections we will give examples 
where Fomin and Greene theory is not applicable.
\end{remark}

\subsection{NSym and QSym}

For the theory of Pieri operators as developed in~\cite{bmws} we need to introduce two graded dual Hopf algebras.
First, the algebra of non-commutative symmetric functions $Nsym$ is a non-commutative analogue of $Sym$ that arises by
considering an algebra with one non-commutative generator at each positive
degree. 
We define $Nsym$ as the algebra with generators $\{\HH_1, \HH_2, \dots \}$ and 
no relations. Each generator $\HH_i$ is defined to be of degree $i$, 
giving $Nsym$ the structure of a graded algebra. We let $Nsym_n$ denote the 
graded component of $Nsym$ of degree $n$. A basis for $Nsym_n$ is given by the set of
\textit{complete homogeneous functions} 
$\{\HH_\alpha := \HH_{\alpha_1} \HH_{\alpha_2} \cdots \HH_{\alpha_m}\}_{\alpha \vDash n}$ 
indexed by compositions $\alpha$ of $n$.  

We have the projection morphism $\chi\colon Nsym \to Sym$ defined by sending the basis element  
$\HH_\alpha$ to the complete homogeneous symmetric function 
 $$\chi(\HH_\alpha) := h_{\alpha_1} h_{\alpha_2} \cdots h_{\alpha_{\ell(\alpha)}}$$
 and extended linearly to all of $Nsym$. 
A second basis of $NSym$ is given by the $R_\alpha$, usually called the \textit{ribbon basis}. The $R_\alpha$ are defined by 
\begin{equation}\label{def:Ribbon} R_\alpha = \sum_{\beta \geq \alpha} (-1)^{\ell(\alpha)-\ell(\beta)} \HH_\beta
\hskip .2in \textrm{or equivalently} \hskip .2in \HH_\alpha = \sum_{\beta \geq \alpha} R_\beta. \end{equation}

The product expansion follows easily from the non-commutative product on the generators 
\[\HH_\alpha \HH_\beta = 
\HH_{\alpha_1, \ldots \alpha_{\ell(\alpha)},\beta_1, \ldots \beta_{\ell(\beta)}}~.\]
$Nsym$ has a coalgebra structure, which is defined on the generators by
\[ \Delta( \HH_j ) = \sum_{i=0}^j \HH_i \otimes \HH_{j-i}~. \]
This determines the action of the coproduct on the basis $\HH_\alpha$ since the
coproduct is an algebra morphism with respect to the product. 

Second, the Hopf algebra of quasi-symmetric functions, $Qsym$ is dual to $Nsym$ and contains $Sym$ as a subalgebra. 
The  graded  component  $Qsym_n$  is  indexed  by  compositions  of  $n$. 
This algebra is most readily realized within the 
ring of power series of bounded degree 
$\mathbb{Q}[\![x_1, x_2, \dots]\!]$. The monomial 
quasi-symmetric function indexed by a composition $\alpha$ is defined as
\begin{equation}
    \label{monomial-qsym}
    M_\alpha = \sum_{i_1 < i_2 < \cdots < i_m} x_{i_1}^{\alpha_1} x_{i_2}^{\alpha_2} \cdots x_{i_m}^{\alpha_m}.
\end{equation}
The algebra of quasi-symmetric functions, $Qsym$, can then be defined as the algebra with the 
monomial quasi-symmetric functions as a basis, whose multiplication is inherited as a subalgebra of $\mathbb{Q}[\![x_1, x_2, \dots]\!]$.
We define the coproduct on this basis as:
\[ \Delta(M_\alpha) = \sum_{S \subset \{1,2, \ldots, \ell(\alpha)\}} M_{\alpha_{S}} \otimes M_{\alpha_{S^c}},\]
where if $S=\{ i_1 < i_2 < \cdots < i_{|S|}\}$, then 
$\alpha_S = [\alpha_{i_1}, \alpha_{i_2}, \ldots, \alpha_{i_{|S|}}]$.

We view $Sym$ as a subalgebra of $Qsym$. In fact, the usual monomial symmetric functions $m_\lambda \in Sym$
expand positively in the quasi-symmetric monomial functions :
\[ m_\lambda = \sum_{sort(\alpha) = \lambda} M_\alpha,\]
where $sort(\alpha)$ is the partition obtained by organizing the parts of $\alpha$ from the largest to the smallest.

The \textit{fundamental quasi-symmetric functions}, denoted by $F_\alpha$ form another basis of $Qsym_n$ and are defined by their expansion
in the monomial quasi-symmetric basis: 
\[F_\alpha = \sum_{\beta \leq \alpha} M_\beta.\]

The algebras $Qsym$ and $Nsym$ form graded dual Hopf algebras. The monomial basis of $Qsym$ is dual in this context to the complete homogeneous basis of $Nsym$, and the fundamental basis of $Qsym$ is dual to the ribbon basis of $Nsym$.
$Nsym$ and $Qsym$ have a pairing $\langle \cdot, \cdot \rangle: Nsym \times Qsym \to \mathbb{Q}$, defined under this duality as either $\langle \HH_\alpha, M_\beta \rangle = \delta_{\alpha, \beta}$, or $\langle R_\alpha, F_\beta \rangle = \delta_{\alpha, \beta}$.

\subsection{Skew function $K_{[\lambda,\nu]}$} \label{sec:K}
Associated to any $\lambda\subseteq\nu$ in $\mathcal Y$, we construct a quasisymmetric function $K_{[\lambda,\nu]}$ following the notion of Pieri operators as developed in~\cite{bmws}. Let $\langle \lambda,\mu \rangle=\delta_{\lambda,\mu}$ define a scalar product on ${\mathbb Z}{\mathcal Y}$.
Using the operators $H_k$ on ${\mathbb Z}{\mathcal Y}$ we can define
  $$K_{[\lambda,\nu]} = \sum_\alpha \langle H_\alpha(\lambda),\nu\rangle M_\alpha. $$
In view of the commutation relation $H_aH_b=H_bH_a$, the function $K_{[\lambda,\nu]} $ is not only  quasisymmetric but symmetric as well. Indeed
since $H_\alpha=H_{sort(\alpha)}$ and since $m_\lambda = \sum_{sort(\alpha) = \lambda} M_\alpha$,  we have that
  $$K_{[\lambda,\nu]} = \sum_\mu \langle H_\mu(\lambda),\nu\rangle m_\mu $$
is symmetric. We are interested in knowing the coefficients of $K_{[\lambda,\nu]}$ when expanded in different bases. 
We remark that we have an action of $NSym$ on ${\mathbb Z}{\mathcal Y}$ given by $\HH_{\alpha}.\lambda = H_\alpha(\lambda)$.
In this case the action factors through the projection $\chi\colon NSym\to Sym$. As observed earlier, the basis $\HH_\alpha$ of $NSym$ is dual to the basis $M_\alpha$ of $QSym$.
A straightforward computation shows that
  $$K_{[\lambda,\nu]} = \sum_\alpha \langle \HH_\alpha.\lambda,\nu\rangle M_\alpha  = \sum_\alpha \langle {\bf x}_\alpha.\lambda,\nu\rangle Y_\alpha.$$
for any dual basis ${\bf x}_\alpha$ and $Y_\alpha$ of $NSym$ and $QSym$ respectively. We thus have that

\begin{theorem} \label{thm:Kmulambda}
$$K_{[\lambda,\nu]} = \sum_\alpha \langle R_\alpha.\lambda,\nu\rangle F_\alpha =  \sum_\mu C_{\lambda,\mu}^\nu s_\mu$$
where $C_{\lambda,\mu}^\nu$ is given in~\eqref{eq:LR}. Moreover for $\alpha=(\alpha_1,\ldots,\alpha_k)$ a composition of $n$, we have that
$\langle R_\alpha.\lambda,\nu\rangle$ counts the number of paths in $\mathcal Y$ from $\lambda$ to $\nu$ with labels $i_1,i_2,\ldots,i_n$ such that
$i_r>i_{r+1}$ if and only if $r\in D(\alpha)=\{\alpha_1,\alpha_1+\alpha_2,\ldots,n-\alpha_k\}$.
\end{theorem}

\proof The first equality follows from duality between the $R_\alpha$ and the $F_\alpha$. For the second equality, from the definition of $H_k$ we remark that for
$$K_{[\lambda,\nu]} = \sum_\mu \langle H_\mu(\lambda),\nu\rangle m_\mu $$
the coefficient $ \langle H_\mu(\lambda),\nu\rangle = d_{\lambda,\mu}^\nu$ is the coefficient of $s_\nu$ in the product $s_\lambda h_\mu$.
In $Sym$, the basis $h_\mu$ and $m_\mu$ are dual and the basis $s_\mu$ is self dual. Hence the coefficient of $s_\mu$ in $K_{[\lambda,\nu]}$
is the same as the coefficient of $s_\nu$ in $s_\lambda s_\mu$. 

The fact that $\langle R_\alpha.\lambda,\nu\rangle$ counts the paths as described follows from a simple inclusion-exclusion argument and the fact that by definition
$\langle \HH_\alpha.\lambda,\nu\rangle$ counts the paths in $\mathcal Y$ from $\lambda$ to $\nu$ with labels $i_1,i_2,\ldots,i_n$ such that
$i_r>i_{r+1}$ only if $r\in D(\alpha)=\{\alpha_1,\alpha_1+\alpha_2,\ldots,n-\alpha_k\}$.
\qed

\begin{remark} The function $K_{[\lambda,\nu]}$ in Theorem~\ref{thm:Kmulambda} is the well known skew-Schur function $s_{\nu/\lambda}$. It is denoted $F_{\nu/\lambda}$
by Fomin and Greene in~\cite{FG98}. Theorem~1.2 of~\cite{FG98} shows that the coefficients $C_{\lambda,\mu}^\nu$ are positive and count
 paths in $\mathcal Y$ satisfying a precise rule. This is a very powerful method that works for any monoid of operators $\u{r}$ satisfying the relations~\eqref{eq:relFG}.
 Several classical examples are solved by this theory which gives a method to understand the coefficients we are interested in. This includes the weak order of the symmetric group 
 and the Stanley symmetric function $F_{w/u}$ originally defined in~\cite{stan84}.
There are many new situations where  Fomin and Greene theory cannot be applied and we will give some examples of this in the next sections.
\end{remark}


\section{Schubert vs Schur }\label{sc:New}

We present an example of a monoid that does not satisfy  Fomin and Greene's conditions, yet it is interesting and still yields some symmetry and positivity.
In this example, which is taken from the theory of Schubert polynomials (see \cite{bsottileschub,LS82a,Macdonald91}), positivity results are highly non-trivial. We consider operators on the infinite symmetric group defined from Monk's rule.
From these operators one defines Pieri operators that mimic the multiplication of Schubert polynomials by symmetric functions. 
Symmetry follows from the commutativity of multiplication and positivity follows from geometry.
A combinatorial proof of positivity is much harder to obtain and was only recently achieved in~\cite{assafbsottile} using the techniques of~\cite{A_JAMS}.

Let $u\in {\mathcal S_{\infty}}:=\bigcup_{n\ge0} {\mathcal S}_n$ be an infinite permutation where all but a finite number of positive integers are fixed. Schubert polynomials $\mathfrak S_u$ are indexed by such permutations ~\cite{LS82a,Macdonald91}.
These polynomials form a homogenuous basis of the polynomial ring ${\mathbb Z}[x_1,x_2,\ldots]$ in countably many variables. The coefficients $c_{u,v}^w$ in
\begin{equation}\label{eq:genschub}
  {\mathfrak S}_u {\mathfrak S}_v = \sum_{w} c_{u,v}^w {\mathfrak S}_w ,
\end{equation}
are known to be positive from geometry. 

\subsection{Operators on the $r$-Bruhat order} 

We now define operators on the $r$-Bruhat order on $ {\mathcal S_{\infty}}$.
Let
$\ell(w)$ be the length of a permutation $w\in {\mathcal S}_\infty$.
We define the $r$-Bruhat order $<_r$ by its covers.
Given permutations $u,w\in{\mathcal S}_\infty$, we say that
$u\lessdot_r w$ if $\ell(u)+1=\ell(w)$ and
$u^{-1}w=(i,j)$, where $(i,j)$ is a reflection with $i\leq r<j$.

For $0<a<b$, let $\u{\a\b}$ denote the operator on ${\mathbb Z} \S_\infty$ defined by 
\begin{equation}\label{eq:operat}
\begin{array}{rcl}
\u{\a\b} \colon\  {\mathbb Z}\S_\infty&\longrightarrow& \quad{\mathbb Z}\S_\infty,\\
u\quad&\longmapsto&\ \ \rule{0pt}{28pt} \left\{\begin{array}{ll}
(\a\,\,\,\b)u   
&\mbox{ if } u\lessdot_r (\a,\b)u,
\medskip       
\\      0& \mbox{ otherwise.}\end{array}\right.
\end{array} 
\end{equation}
We have shown in~\cite{bsottilemonoid} that these operators satisfy  the following relations:
\begin{equation} \label{eq:relschub}
\begin{array}{clrclll}
(1)&&\u{\b\c}\u{\c\d}\u{\a\c}&\equiv&\u{\b\d}\u{\a\b}\u{\b\c},\hfill&&
        \hbox{if $\a<\b<\c<\d$},\hfill\\
(2)\hfill&& \hfill \u{\a\c}\u{\c\d}\u{\b\c}&\equiv&
\u{\b\c}\u{\a\b}\u{\b\d},\hfill&&
        \hbox{if $\a<\b<\c<\d$},\hfill\\
(3)\hfill&&  \hfill \u{\a\b}\u{\c\d}&\equiv&\u{\c\d}\u{\a\b}, \hfill&&
        \hbox{if $\b<\c$ or $\a<\c<\d<\b$},\hfill\\
(4)\hfill&&  \hfill   \u{\a\c}\u{\b\d}&\equiv& \u{\b\d}\u{\a\c}\ 
\equiv \  {\bf 0},\hfill&&
        \hbox{if $\a\le \b<\c\le\d$},\hfill\\
(5)\hfill&&  \hfill  \u{\b\c}\u{\a\b}\u{\b\c}
&\equiv&\u{\a\b}\u{\b\c}\u{\a\b}\ \equiv\ {\bf 0},\hfill&&
        \hbox{if $\a<\b<\c$}.
\end{array}  
 \end{equation}
The ${\bf 0}$ in relations (4) and (5) mean that no chain in any $r$-Bruhat order can contain such a sequence of transpositions.
On the other hand, relations (1), (2) and (3) are complete and transitively connect any two chains in a given interval $[u,w]_r$.
It is also interesting to notice that the relations are independent of $r$. This is a fact noticed in~\cite{bsottileschub}: a nonempty interval $[u,w]_r$ in the $r$-Bruhat order is isomorphic to a nonempty interval $[x,y]_{r'}$ in an $r'$-Bruhat order as long as $wu^{-1}=yx^{-1}$. It is important to remark that if one fixes $r$, there are in fact more relations than~\eqref{eq:relschub}.
We will clarify this after Proposition~\ref{prop:relSchub}. 
For the moment we assume that the operator $\u{ab}$ acts on the disjoint union of all $r$-Bruhat orders for $r>0$.
Let ${\mathcal M}\langle\u{ab}\rangle$ be the monoid generated by the $\bf 0$ operator and all operators $\u{ab}$ for $a<b$. A consequence of~\cite{bsottilemonoid}  is the following proposition.

\begin{proposition}\label{prop:relSchub}
 ${\mathcal M}\langle\u{ab}\rangle$ is the monoid freely generated by the $\u{ab}$ for $0<a<b\in{\mathbb Z}$ and $\bf 0$
 modulo the relations~\eqref{eq:relschub}.
\end{proposition}

\begin{remark} When we specify a chain $\u{a_nb_n}\cdots\u{a_2b_2}\u{a_1b_1}$ in the interval  $[u,w]_r$, it is understood that this is the actual sequence of operators $(\u{a_nb_n},\ldots,\u{a_2b_2},\u{a_1b_1})$ we are referring to. This is a slight abuse of notation but it simplifies notation and  the context will make it clear.\end{remark}

In fact we can say much more about the monoid ${\mathcal M}\langle\u{ab}\rangle$. Given any $\zeta\in\S_\infty$ we produce a chain in a nonempty interval $[u,w]_r$ for some $r$
as follows. Let $up(\zeta)=\{a:\zeta^{-1}(a)<a\}$. This is a finite set and we can set $r=|up(\zeta)|$. To construct $w$, we sort the elements in $up(\zeta)=\{i_1<i_2<\cdots<i_r\}$ and its complement $up^c(\zeta)={\mathbb Z}_{>0}\setminus up(\zeta)=\{j_1<j_2<\ldots\}$. Next, we put $w=[i_1,i_2,\ldots,i_r,j_1,j_2,\ldots]\in\S_\infty$ and then we let $u=\zeta^{-1}w$. Notice that $u,w$ and $r$ constructed this way depend on $\zeta$. From~\cite{bsottileschub, bsottilemonoid}, we have that $[u,w]_r$ is non-empty and now we want to construct a chain in $[u,w]_r$. This is done recursively as follows: let
\begin{align*}
  a_1 &= u(i_1) \text{ where } i_1=\max\{i\le r:u(i)<w(i)\} \text{ \;\;\; and  }\\
  b_1 &= u(j_1) \text{ where } j_1=\min\{j> r:u(j)>u(i_1)\ge w(j)\}
\end{align*}
then $\u{a_nb_n}\cdots\u{a_2b_2} \u{a_1b_1}$ is a chain in $[u,w]_r$ for any chain $\u{a_nb_n}\cdots \u{a_2b_2}$ in $[(a_1, b_1)u,w]_r$. Here we have that all the other possible chains
in the interval $[u,w]_r$ are obtained from the chain $\u{a_nb_n}\cdots\u{a_2b_2} \u{a_1b_1}$ by sequences of transformations given in the equation~\eqref{eq:relschub}. This means that
the operator $\u{\zeta} = \u{a_nb_n}\cdots\u{a_2b_2} \u{a_1b_1}$
 is well defined, non zero for any $r'\ge r$ and if $\zeta\ne\zeta'$ then $\u{\zeta}\ne\u{\zeta'}$.
For a fix $r$,
  $$\u{\zeta}(u) =  \left\{\begin{array}{ll}
\zeta u   
&\mbox{ if } u<_r  \zeta u,
\medskip       
\\      0& \mbox{ otherwise.}\end{array}\right.
$$

\begin{example} \label{ex:schubert}
Consider $\zeta=[3, 6, 2, 5, 4, 1,...]$ where all other values are fixed. We have that $up(\zeta)=\{3,5,6\}$ and $up^c(\zeta)=\{1,2,4,...\}$. In this case, $r=3$, $w=[3,5,6,1,2,4,...]$ and $u=[1,4,2,6,3,5,...]$. The recursive procedure above produces the chain $\u{23}\u{12}\u{45}\u{26}$ in $[u,v]_3$. We get all other chains by using the relations~(\ref{eq:relschub}):
\begin{equation} \label{eq:axampleschub}
\begin{array}{cccc}
\u{23}\u{12}\u{45}\u{26},&
\u{23}\u{12}\u{26}\u{45},&
\u{23}\u{45}\u{12}\u{26},&
\u{45}\u{23}\u{12}\u{26},\cr
\u{45}\u{13}\u{36}\u{23},&
\u{13}\u{45}\u{36}\u{23},&
\u{13}\u{36}\u{45}\u{23},&
\u{13}\u{36}\u{23}\u{45}.
\end{array}
 \end{equation}
The interval obtained in this case is
$$
 \raise -50pt
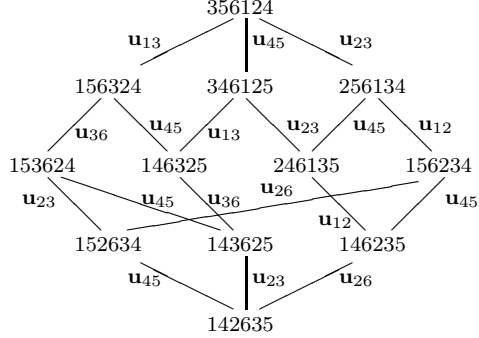
\hbox{ \begin{picture}(200,120)
      \put(100,0){$\scriptstyle 142635$}
      \put(50,30){$\scriptstyle 152634$}
      \put(100,30){$\scriptstyle 143625$}
      \put(150,30){$\scriptstyle 146235$}
      \put(25,60){$\scriptstyle 153624$}
      \put(75,60){$\scriptstyle 146325$}
      \put(125,60){$\scriptstyle 246135$}
      \put(175,60){$\scriptstyle 156234$}
      \put(50,90){$\scriptstyle 156324$}
      \put(100,90){$\scriptstyle 346125$}
      \put(150,90){$\scriptstyle 256134$}
      \put(100,120){$\scriptstyle 356124$}
      \put(120,8){\line(2,1){35}} \put(110,8){\line(-2,1){35}}  \put(115,8){\line(0,1){20}}
      \put(70,18){$\scriptstyle \u{45}$}  \put(117,18){$\scriptstyle \u{23}$}  \put(150,18){$\scriptstyle \u{26}$}
      \put(170,38){\line(1,1){20}} \put(160,38){\line(-1,1){20}}  
      \put(142,40){$\scriptstyle \u{12}$}  \put(190,48){$\scriptstyle \u{45}$}  
      \put(110,38){\line(-1,1){20}} \put(105,38){\line(-3,1){60}}  
      \put(100,48){$\scriptstyle \u{36}$}  \put(75,48){$\scriptstyle \u{45}$}  
      \put(60,38){\line(-1,1){20}} \put(70,38){\line(6,1){110}}  
      \put(30,48){$\scriptstyle \u{23}$}  \put(120,52){$\scriptstyle \u{26}$}  
      \put(140,68){\line(1,1){20}} \put(135,68){\line(-1,1){20}}  
      \put(130,76){$\scriptstyle \u{23}$}  \put(155,76){$\scriptstyle \u{45}$}  
      \put(90,68){\line(1,1){20}} \put(85,68){\line(-1,1){20}}  
      \put(78,76){$\scriptstyle \u{45}$}  \put(100,73){$\scriptstyle \u{13}$}  
      \put(40,68){\line(1,1){20}} 
        \put(50,73){$\scriptstyle \u{36}$}  
      \put(185,68){\line(-1,1){20}}  
        \put(180,76){$\scriptstyle \u{12}$} 
      \put(110,118){\line(-2,-1){35}} \put(120,118){\line(2,-1){35}}  \put(115,118){\line(0,-1){20}}
      \put(70,108){$\scriptstyle \u{13}$}  \put(117,108){$\scriptstyle \u{45}$}  \put(150,108){$\scriptstyle \u{23}$}
     \end{picture}} 
$$
 Since $u <_r \zeta u$ in this case, we have $\u{\zeta}(u) = \zeta u\ne 0$ for this $r$. Now for any $r'\ge r$, we can build $w'=[3,5,6,7,8,\ldots, 7+r'-r,1,2,4,...]$, $u'=[1,4,2,7,8,\ldots, 7+r'-r,,6,3,5,...]$ by adding fixed points of $\zeta=wu^{-1}$ before the position $r'$. In this way we construct a permutation $u'$ such that $u'<_{r'}\zeta u'$ and $\u{\zeta}(u')=\zeta u'\ne 0$ for any $r'\ge r$.
\end{example}

The above discussion shows the following corollary:

\begin{corollary}\label{cor:fixr}
 The monoid  ${\mathcal M}\langle\u{ab}\rangle$ is precisely
  $$ {\mathcal M}\langle\u{ab}\rangle=\big\{  \u{\zeta} : \zeta\in  {\mathcal S_{\infty}} \big\}\cup\big\{ {\bf 0}\big\}.$$
Moreover, if we let ${\mathcal M}_r\langle\u{ab}\rangle$ be the monoid generated by the operator $\u{ab}$ acting on $r$-Bruhat order for a fixed $r$,
we have
   $$ {\mathcal M}_r\langle\u{ab}\rangle=\big\{  \u{\zeta} : \zeta\in  {\mathcal S_{\infty}, \  |up(\zeta)|\le r} \big\}\cup\big\{ {\bf 0}\big\}.$$
\end{corollary}

Here the multiplication in ${\mathcal M}\langle\u{ab}\rangle$ is given by $ \u{\zeta} \u{\eta}= \u{\zeta\eta}$ if $\eta u <_r \zeta\eta u$ for some $u$ and $r$, and is $\bf 0$ otherwise.

\subsection{Pieri operators on $r$-Bruhat order}
We now introduce some Pieri operators related to the operators $\u{ab}$. These Pieri operators are defined in such a way that they mimic the multiplication of a Schubert polynomial by the homogeneous symmetric polynomial $h_k(x_1,\ldots,x_r)$.

A permutation $v\in{\mathcal S_{\infty}}$ such that $v(1)<v(2)<\cdots <v(r)$ and $v(r+1)<v(r+2)<\cdots$ is called $r$\emph{-grassmannian}. 
Any partition $\lambda=\lambda_1\ge\lambda_2\ge\cdots\ge\lambda_r\ge 0$ with at most $r$ non-zero parts defines a unique $r$-grassmannian permutation
  $$v(\lambda, r)=[\lambda_r+1, \lambda_{r-1}+2,\ldots,\lambda_1+r,v(r+1),\ldots],$$
  where $v(r+1)<v(r+2)<\cdots$ are the positive integers not in $\{\lambda_r+1, \lambda_{r-1}+2,\ldots,\lambda_1+r\}$. 
As seen in~\cite{LS82a,Macdonald91}, for any such partition $\lambda$ we have that the Schur polynomial $S_\lambda(x_1,x_2,\ldots,x_r)$ is equal to
the Schubert polynomial 
  $${\mathfrak S}_{v(\lambda,r)} = S_\lambda(x_1,x_2,\ldots,x_r).$$
In particular, the homogeneous polynomial $h_k(x_1,\ldots,x_r)$ is the Schubert polynomial ${\mathfrak S}_{v((k),r)}$.
The multiplication of an arbitrary Schubert polynomial by $h_k(x_1,\ldots,x_r)$ is known as the Pieri formula for Schubert polynomials.
It was originally stated as a theorem by Lascoux and Sch\"utzenberger~\cite{LS82a}
with a very brief outline of a proof.  Sottile later proved this
formula geometrically and clarified the history for us~\cite{Sottile}. 
Using the operators $\u{ab}$ on the $r$-Bruhat order, this
can be stated as follows.
\begin{equation}\label{eq:PieriSchub}
  {\mathfrak S}_u h_k(x_1,\ldots,x_r) =   {\mathfrak S}_u {\mathfrak S}_{v((k),r)}= \sum_{w}  {\mathfrak S}_w ,
\end{equation}
where the sum is over all $w>_r u$ such that $ \u{a_kb_k}\cdots\u{a_2b_2} \u{a_1b_1}(u)=w$ for some $b_1<b_2<\cdots <b_k$. It is known (see~\cite{bsottileschub,bsottilemonoid}) that in such interval $[u,w]_r$,
there must be a chain from $u$ to $w$ that is increasing in the sense that $ \u{a_kb_k}\cdots\u{a_2b_2} \u{a_1b_1}(u)=w$ with $b_1<b_2<\cdots <b_k$. Such a chain, when it exists, is unique among all saturated chains in  $[u,w]_r$.

We now introduce series $H_k$ similar to Section~\ref{subsection:ClassicPieri} that will commute with each other and encode the Pieri formula for Schubert polynomials.
Let
 \begin{equation}\label{eq:PieriopSchub}
 H_k= \sum_{b_1<b_2<\ldots<b_k \atop a_i < b_i} \u{a_kb_k}\cdots\u{a_2b_2}\u{a_1b_1}.
 \end{equation}
 Many of the terms in this sum are zero, the non-zero terms have a very special form. In~\cite{LS82a}, we see that it is important to look at the disjoint decomposition of $\zeta$ into disjoint cycles.
In the next proposition we describe the $\u{\zeta}$ appearing in $H_k$ and the structure of the disjoint cycles. For $\zeta\in {\mathcal S_{\infty}}$, let $\zeta=C_1C_2\cdots C_s$ be the decomposition of $\zeta$ in disjoint non-trivial cycles. There are only finitely many non-fixed points, so only finitely many non-trivial cycles. Given a cycle $C=(c_1,c_2,\ldots,c_m)$, we say that $C$ is {\sl increasing} if $c_m<c_{m-1}<\cdots<c_{1}$. Given two disjoint increasing cycles $C=(c_1,c_2,\ldots,c_m)$ and $C'=(c'_1,c'_2,\ldots,c'_n)$ we say that they are {\sl totally disjoint} if
any of the following happens
\begin{enumerate}
 \item $[c_m,c_{1}]\cap [c'_n,c'_{1}] = \emptyset$, or
 \item $[c_m,c_{1}]\cap \{c'_1,c'_2,\ldots,c'_n\} =\emptyset$, or 
 \item $[c'_n,c'_{1}]\cap \{c_1,c_2,\ldots,c_m\} =\emptyset$.
\end{enumerate}
In case (1), the two cycles have support in disjoint intervals. In cases (2) and (3), If the intervals intersect, their intersection must fall all between two sucessive elements in the support of the other cycles. For $C=(c_1,c_2,\ldots,c_m)$ let $|\!|C|\!|=m-1$.
For $\zeta=C_1C_2\cdots C_s$  a product of totally disjoint increasing cycles such that $k=\sum_{i=1}^s |\!|C_i|\!|$, we say that $\zeta$ is {\sl $k$-increasing}.

\begin{proposition}\label{prop:cyclpieri}
   $$H_k= \sum_{\zeta} \u{\zeta},$$
   where $\zeta$ runs over $k$-increasing permutations.
   \end{proposition}
\begin{proof}
  We proceed by induction on $k$. The result is clear for $k=1$. Assume the result is true for any non-zero product $\u{\zeta'}=\u{a_{k-1}b_{k-1}}\cdots\u{a_2b_2}\u{a_1b_1}$ such that
  $b_1<b_2<\ldots<b_{k-1}$. We assume that $\u{\zeta'}=\u{C_1}\u{C_2}\cdots\u{C_s}$ where $\zeta=C_1C_2\cdots C_s$ are totally disjoint increasing cycles.
  For $C=(c_1,c_2,\ldots,c_m)$, an increasing cycle, we have $\u{C}=\u{c_2c_1}\u{c_3c_2}\cdots\u{c_mc_{m-1}}$. A careful analysis of the relation~\eqref{eq:relschub} shows
  that for totally disjoint increasing cycles $C_1C_2\cdots C_s$, the operators $\u{C_i}$ and $\u{C_j}$ commute for $i\ne j$. We will assume that $a_{k-1}$ and $b_{k-1}$ belong to the cycle $C_1$.
  
  We investigate what happens when we perform a non-zero product $\u{a_kb_k}\u{\zeta'}$ where $b_k>b_{k-1}$. If $a_k>b_{k-1}$, then $(b_k,a_k)$ is a new increasing cycle totally disjoint from any cycle of $\zeta'$. If $a_k=b_{k-1}$, then $a_k,b_k$ increases the cycle $C_1$ of $\zeta'$ and is still totally disjoint from the other cycles of $\zeta'$. 
   
If $a_k<b_{k-1}$, then from~\eqref{eq:relschub}-(4) we must have $a_{k}<a_{k-1}$ and $\u{a_{k}b_{k}}\u{a_{k-1}b_{k-1}}\ne {\bf 0}$ commute. Let $C_1=(c_1,c_2,\ldots,c_m)$
   and recall that we have $b_{k-1}=c_1$ and $a_{k-1}=c_2$. We have $\u{C_1}=\u{c_2c_1}\u{c_3c_2}\cdots\u{c_mc_{m-1}}$ and $b_k>b_{k-1}=c_1>c_i$ for all $i$.
   Since $a_k<a_{k-1}=c_2$, then $\u{a_kb_k}\u{c_3c_2}\ne {\bf 0}$ implies $a_k<c_3$ and $\u{a_kb_k}\u{c_3c_2}$ commutes. Continuing this process, we find that $\u{a_kb_k}\u{C_1}=\u{C_1}\u{a_kb_k}\ne{\bf 0}$ and $a_k<c_m<c_1<b_k$. This means $C_1$ and $(b_k,a_k)$ are totally disjoint increasing cycles. 
We have
  $$\u{a_kb_k}\u{\zeta'}=\u{C_1}\u{a_kb_k}\u{C_2}\cdots\u{C_s}$$
From the induction hypothesis, the result holds for $\u{a_kb_k}\u{C_2}\cdots\u{C_s}$ and decomposes into totally disjoint increasing cycles. Moreover $C_1$ will be totally disjoint from the cycles
of $(b_k a_k)C_2\cdots C_s$.
\end{proof}

As in Corollary~\ref{cor:fixr}, the expression in Proposition~\ref{prop:cyclpieri} is valid as long as we consider all possible $r$-Bruhat orders for $r>1$.
If we fix $r$, then most of the $\u{\zeta}$ in $H_k$ will act as zero on the $r$-Bruhat order.
For a fixed $r$, we see that $H_k\colon {\mathbb Z} {\mathcal S}_\infty \to  {\mathbb Z} {\mathcal S}_\infty$ is a well defined operator on the $r$-Bruhat order.
From Corollary~\ref{cor:fixr}, for a fixed $r$, 
   $$H_k= \sum_{\zeta  \mbox{ \tiny  is $k$-increasing} \atop |up(\zeta)|\le r} \u{\zeta}.$$
By definition of $H_k$ and equation~\eqref{eq:PieriSchub}, we have 
  $$ H_k(w)=\!\!\!\! \sum_{wu^{-1} \mbox{ \tiny  $k$-increasing}} \!\!\!\!  u  \qquad \qquad \iff \qquad \qquad {\mathfrak S}_w h_k(x_1,\ldots, x_r) 
  =\!\!\!\!  \sum_{wu^{-1} \mbox{ \tiny $k$-increasing}} \!\!\!\!  {\mathfrak S}_u\,.$$
This implies that
\begin{equation}\label{eq:symH}
 \begin{array}{lr}
  H_b H_a(w) = \sum_{\nu} d_{w,(a,b)}^u u \hfill &\cr
 \qquad \qquad   \iff  &{\mathfrak S}_w h_a(x_1,\ldots, x_r)  h_b(x_1,\ldots, x_r)  =\sum_{u} d_{w,(a,b)}^u {\mathfrak S}_u\,.\cr
 \end{array}
 \end{equation}
In particular, for all $w$ we have  $H_b H_a(w)=H_a H_b(w)$ since $h_ah_b =h_b  h_a $. The result below is not as
well known as Theorem~\ref{thm:youngop}. 

\begin{theorem}\label{thm:rBruhatop}
 The algebra  ${\bf B}\langle H_{k}\rangle$ spanned by $\{H_1, H_2, H_3,\ldots\}$ as operators on the $r$-Bruhat order for $r>0$ is isomorphic to $Sym$.
\end{theorem}

\proof
As we multiply $H_aH_b$ and $H_bH_a$, some terms will go to zero and others will survive. The terms that survive in  $H_aH_b$  are  of the form 
 $$ \u{w} = \u{\zeta}\u{\eta}$$
 where $\zeta$ is $a$-increasing and $\eta$ is $b$-increasing. Let $d_{(a,b)}^w$ be the coefficient of $\u{w}$ in $H_aH_b$. From Corollary~\ref{cor:fixr}, for any $w\in{\mathcal S}_\infty$
 we can find $u$ and an $r>0$ such that $\u{w}(u)=v\ne 0$. So $d_{(a,b)}^w$ is the coefficient of $v$ in $H_aH_b(u)$. From~\eqref{eq:symH}, for all $w$, we have 
  $$ d_{(a,b)}^w = \mbox{Coeff of $v$ in $H_aH_b(u)$} = \mbox{Coeff of $v$ in $H_bH_a(u)$} = d_{(b,a)}^w. $$
 Hence $H_aH_b=H_bH_a$.
 
The algebra ${\bf B}\langle H_{k}\rangle$ is clearly spanned by $H_\lambda=H_{\lambda_1}\cdots H_{\lambda_\ell}$ where $\lambda$ runs over all partitions.
To show the isomorphism with $Sym$, we only need to show that the $H_\lambda$ are linearly independent. Let $r\ge \ell(\lambda)$.
Using \eqref{eq:symH} we have that
$H_\lambda(Id)=\sum_\mu d_{\lambda}^{\mu} v_\mu$  where $v_\mu$ is the unique grassmannian permutation defined by $v(v_\mu,r)=\mu$

and the $d_{\lambda}^{\mu}$ satisfy
$$h_\lambda(x_1,\ldots,x_r)=\sum_\mu d_{\lambda}^{\mu} s_\mu(x_1,\ldots,x_r).$$
If we have a finite linear combination $\Phi = \sum_\lambda c_\lambda H_\lambda$, then for $r\ge \max\{ \ell(\lambda):c_\lambda\ne 0\}$ we have that $\Phi(Id)$ corresponds to the symmetric function $\sum_\lambda c_\lambda h_\lambda$. This is zero if and only if all $c_\lambda=0$.
\qed

As in Section~\ref{sec:K}, let $\langle v,w\rangle=\delta_{v,w}$ define a scalar product on ${\mathbb Z}{\mathcal S}_\infty$. 
For a fixed $r>0$ and $u<_r w$, we  define the quasisymmetric function
\begin{equation}\label{eq:Kuw}  
    K_{[u,w]_r} =\sum_\alpha \langle H_\alpha(u),w\rangle M_\alpha.
\end{equation}
As before, since $H_aH_b=H_bH_a$, the function $K_{[u,w]_r} $ is in fact a symmetric function. As shown in \cite{bsottileskew,bmws}, we have the following theorem:
\begin{theorem}\label{THM:Kuw}
    $$K_{[u,w]_r} = \sum_\alpha \langle R_\alpha(u),w \rangle F_\alpha = \sum_\mu c_{u,v(\mu,r)}^w s_\mu\,,$$
    where $ c_{u,v(\mu,r)}^w$ are defined in~\eqref{eq:genschub}. Moreover for $\alpha$, a composition of $n$, we have that $\langle R_\alpha(u),w \rangle$ counts the number of paths in the r-Bruhat order ${\mathcal S}_\infty$ from $u$ to $w$ of the form $\u{a_nb_n}\cdots\u{a_2b_2}\u{a_1b_1}$ where $b_i>b_{i+1}$ if and only if $i\in D(\alpha)$.
\end{theorem}

\begin{example}\label{ex:K}
Using the chains in~(\ref{eq:axampleschub}) and Theorem~\ref{THM:Kuw} we can compute the quasisymmetric function associated to this interval and we get
$$
 \begin{array} {rcl}
   K_{[142635,356124]_3} &=& F_{13} + F_{121} + F_{22} + F_{112} + F_{121} + F_{31} + F_{211} + F_{22}\cr
                                          &=& S_{31} + S_{22} + S_{211}.
 \end{array}$$
 \end{example}

\begin{remark}
  The monoid generated by the operators $\u{ab}$ does not satisfy relations that resemble~\eqref{eq:relFG}, hence we cannot use the work of Fomin and Greene to
  conclude that $K_{[u,w]_r}$ is symmetric nor deduce a combinatorial rule for constructing the coefficient $ c_{u,v(\mu,r)}^w$ in $K_{[u,w]_r}$.
  In fact all known attempts to give such a rule  so far have failed. 
  In the next section we outline how it is shown combinatorially in~\cite{ABS} that the coefficients are positive (without giving an explicit rule in all cases) using techniques developed by~\cite{A_JAMS}.
\end{remark}

\subsection{Combinatorial proof of positivity of $c_{u,v(\mu,r)}^w$}\label{sec:pos}
Let ${Comp}_n$ denote the set of compositions of $n$.
Given a finite family of objects $\mathcal C$ and a function $\alpha\colon {\mathcal C}\to Comp_n$ we can define a quasisymmetric function as follows
  $$ K_{\mathcal C} = \sum_{x\in {\mathcal C}} F_{\alpha(x)} \,.$$
The function $K_{[u,w]_r}$ of Theorem~\ref{THM:Kuw} is clearly of this form. In that case ${\mathcal C}$ is the set of saturated chains $\u{a_nb_n}\cdots\u{a_2b_2}\u{a_1b_1}$ in the interval $[u,w]_r$
and $\alpha=\alpha(\u{a_nb_n}\cdots\u{a_2b_2}\u{a_1b_1})$ is the unique composition where $b_i>b_{i+1}$ if and only if $i\in D(\alpha)$.

Assaf~\cite{A_JAMS} develops new combinatorial techniques to show that quasisymmetric functions of the form $K_{\mathcal C}$ are symmetric with a positive expansion in terms of Schur functions.
To this end one must construct partially commuting involutions $\phi_i\colon {\mathcal C}\to{\mathcal C}$ for $1<i<n$ satisfying a set of axioms.
When $\mathcal C$ consists of words (or saturated chains), the involutions $\phi_i$ can be viewed as an analogue of the dual Knuth relations.
In~\cite{ABS} we have defined such involution $\phi_i$ on the set of chains of $[u,w]_r$.
Given a chain $x=\u{a_nb_n}\cdots\u{a_2b_2}\u{a_1b_1}$, the involution $\phi_i$ will only affect the three entries $\u{a_{i+1}b_{i+1}}\u{a_ib_i}\u{a_{i-1}b_{i-1}}$.
We set $\phi_i(x)=x$ if and only if $\big| D(\alpha(x)) \cap \{i-1,i\}\big| \ne 1$. When $\big| D(\alpha(x)) \cap \{i-1,i\}\big| = 1$, the entries $\u{a_{i+1}b_{i+1}}\u{a_ib_i}\u{a_{i-1}b_{i-1}}$
of $x$ can be one of twelve cases. To define $\phi_i$, we match the twelves cases as follows:
\begin{enumerate}
\item[{\bf (A)}]  $\u{\gamma c}\u{\alpha a}\u{\beta b}\leftrightarrow\u{\alpha a}\u{\gamma c}\u{\beta b}$,
\item[\phantom{\bf (A')}] $\u{\beta b}\u{\alpha a}\u{\gamma c} \leftrightarrow\u{\beta b}\u{\gamma c}\u{\alpha a}$,\qquad
         if $\{a,\alpha\}\cap\{c,\gamma\}=\emptyset$ and $a<b<c$,
\item[{\bf (B)}] $\u{b c}\u{a b}\u{b d} \leftrightarrow\u{a c}\u{c d}\u{b c}$,
\item[\phantom{\bf (B')}] 
$\u{b d}\u{a b}\u{b c} \leftrightarrow\u{b c}\u{c d}\u{a c}$,\hskip2.75em
         if $a<b<c<d$,
\item[{\bf (C)}] $\u{\beta b}\u{\alpha a}\u{a c} \leftrightarrow \u{\alpha a}\u{a c}\u{\beta b}$,
\item[\phantom{\bf (C')}] $\u{ac}\u{\alpha a}\u{\beta b} \leftrightarrow \u{\beta b}\u{ac}\u{\alpha a}$,\hskip2.4em
         if $\{\alpha,a,c\}\cap\{b,\beta\}=\emptyset$ and $a<b<c$.
\end{enumerate}
This matching is completely determined by the relations in~\eqref{eq:relschub}. We see them as the analogue of the dual Knuth relations for this problem. 
Instead of using the relation~\eqref{eq:relschub} one can investigate the free monoid spanned by the $\u{ab}$ modulo the dual Knuth relations above.
Under certain axioms described in~\cite{A_JAMS,ABS}, the component of the equivalent classes of these relations will be combinatorially symmetric and Schur positive.
To our knowledge this is the best we can do so far, and is the best generalization of the work of Fomin and Greene.

\section{$k$-Schur functions}\label{sc:future}
In this section we present a monoid of operators for which much less is known but that is expected to behave as in Section~\ref{sc:New}.
This monoid is related to the so-called $k$-Schur functions~\cite{LLM03,LLMS}.
This time we will define operators on the Bruhat order of the $k$-affine symmetric group. 
The operators we define will be related to the multiplication of dual $k$-Schur functions.
There are still many open questions in this case, but we will present our program and we believe that it can be solved in the same spirit as in Section~\ref{sc:New}.
There is another order one may consider on the $k$-affine symmetric group, namely the {\sl weak} order.
The operators corresponding to the weak order are related to the multiplication of $k$-Schur functions, but we will discuss only briefly  the difficulties which arise in this situation.

The $k$-Schur functions were originally defined combinatorially in terms of
$k$-atoms, and conjecturally provide a positive decomposition of the Macdonald
polynomials~\cite{LLM03}. These functions have several definitions and it is conjectural that they are equivalent (see~\cite{LLMS}). In this paper we will adopt the definition given by the $k$-Pieri rule
and $k$-tableaus (see~\cite{lapointemorseDEF,LLMS}) since this gives us a relation with the homology and cohomology of the affine grassmannians and we therefore get positivity in their structure constants.

Different objects index $k$-Schur functions: $0$-grassmannian in $k$-affine permutations, $k+1$-cores, $k$-bounded partitions. Originally (as in~\cite{LLM03}), $k$-Schur functions were indexed by $k$-bounded partitions $\lambda=(\lambda_1,\lambda_2,\ldots,\lambda_\ell)$ where $\lambda_1\le k$. These partitions are in bijection with $k+1$-cores (see ~\cite{LM1}). By definition, $k+1$-cores are integer partitions $\mu=(\mu_1,\mu_2,\ldots,\mu_m)$ with no hook of lenght $k+1$. To close the loop,  in~\cite{bjornerbrenti} it is shown that $k+1$-cores are in bijection with $0-$grassmannian permutations in the $k$-affine symmetric group (see also~\cite{BBTZ, LLMS}). 

\subsection{Affine symmetric group.}
The $k$-affine symmetric group $W=\tilde{A}_k$ is generated by reflections $s_i$ for $i \in \{0,1,\ldots,k\}$, subject to the relations: 
 $$		s_i^2 = 1 ;\qquad
		s_is_{i+1}s_i = s_{i+1}s_is_{i+1};\qquad
		s_is_j = s_js_i  \ \textrm{ if } i-j \neq \pm 1,
 $$
 where $i-j$ and $i+1$ are understood to be taken modulo $k+1$.
 Let $w\in W$ and  denote its length by $\ell(w)$, given by the minimal number of generators needed to write a reduced expression for $w$. We let $W_0$ denote the parabolic subgroup obtained from $W$ by removing the generator $s_0$. This is naturally isomorphic to the symmetric group $\S_{k+1}$.
For more details on the affine symmetric group see \cite{bjornerbrenti}.

Let $u\in W$ be an affine permutation. This permutation can be represented using window notation. That is, $u$ can be seen as a bijection from $\mathbb Z$ to $\mathbb Z$, so that if $u_i$ is the image of the integer $i$ under $u$, then it can be seen as a sequence:
$$
u=\cdots | u_{-k}\; \cdots\; u_{-1}\; u_{0}\underbrace{| u_1\; u_2\; \cdots\; u_{k+1} |}_{\text{main window}}u_{k+2}\; u_{k+3}\; \cdots\; u_{2k+2}|\cdots
$$
Moreover, $u$ satisfies the property that  $u_{i+k+1}=u_i+k+1$ for all $i$, and the sum of the entries in the main window $u_1+u_2+ \cdots+ u_{k+1}={{k+2}\choose{2}}$. Notice that in view of the first property, $u$ is completely determined by the entries in the main window. 
In this notation, the generator $u=s_i$ is the permutation such that $u_{i+m(k+1)}=i+1+m(k+1)$ and $u_{i+1+m(k+1)}=i+m(k+1)$ for all $m$, and $u_j=j$ for all other values. The multiplication $uw$ of permutations $u,w$ in $W$ is the usual composition given by $(uw)_i = u_{w_i}$. In view of this, the parabolic subgroup $W_0$ corresponds to the $u\in W$ such that the numbers $\{1,2,\ldots,k+1\}$ appear in the main window. 

Now, let $W^0$ denote the set of minimal length coset representatives of $W/W_0$. In this paper we take right coset representatives, although left coset representatives could be also taken. The set of permutations in $W^0$ are the \emph{affine grassmannian permutations} of $W$, or $0$-grassmannians for short. 

\begin{definition}\label{grassmannian}
The {\sl affine $0$-grassmannians} $W^0$ are  the permutations $u\in W$ such that the numbers $1,2,\ldots,k+1$ appear from left to right in the sequence $u$.\end{definition}

\begin{example}\label{5core}
Let $k=4$ and 
$$
u=\cdot\cdot\cdot | \bar{3}\;\bar{2}\;1\;\bar{5}\;\bar{1}\underbrace{| 2\;3\;6\;\bar{0}\;4|}_{\text{main window}}7\;8\;11\;5\;9|\cdot\cdot\cdot 
$$
where $\bar{i}$ stands for $-i$. By convention we say that $0$ is negative.
This permutation $u$ is $0$-grassmannian and it corresponds to the $5$-core $\mu=(4,1,1)$. The correspondence
is easy to see from the window notation. We just need to read the sequence of entries of $u$, drawing a vertical step down for each negative entry,
and an horizontal step right for each positive entry. The result is the diagram of $\mu$:
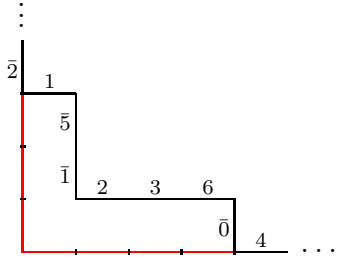
$$
 \raise -0pt\hbox{ \begin{picture}(100,100)
 \rouge{   \put(0,0) {\line(1,0){80}}
               \put(0,0) {\line(0,1){60}} }
 \put(-1,20) {\line(1,0){2}}    \put(-1,40) {\line(1,0){2}}     \put(-1,60) {\line(1,0){2}}   
 \put(20,-1) {\line(0,1){2}}    \put(40,-1) {\line(0,1){2}}     \put(60,-1) {\line(0,1){2}}    \put(80,-1) {\line(0,1){2}}
 \put(0,60) {\line(0,1){20}}   \put(0,60) {\line(1,0){20}}  \put(20,60) {\line(0,-1){40}} \put(20,20) {\line(1,0){60}}
  \put(80,20) {\line(0,-1){20}} \put(80,00) {\line(1,0){20}}
  \put(-2,85){$\vdots$}    \put(105,0){$\ldots$}
  \put(-6,66){$\scriptstyle \bar{2}$}  
  \put(8,62){$\scriptstyle {1}$}  
  \put(14,46){$\scriptstyle \bar{5}$}    \put(14,26){$\scriptstyle \bar{1}$}  
  \put(28,22){$\scriptstyle {2}$}    \put(48,22){$\scriptstyle {3}$}    \put(68,22){$\scriptstyle {6}$}  
  \put(74,6){$\scriptstyle \bar{0}$}  
  \put(88,2){$\scriptstyle {4}$}     
      \end{picture}} 
$$
\end{example}

\subsection{$k$-Schur functions and weak order. }
As previously  mentioned, $0$-grassmannian permutations index $k$-Schur functions, which we  denote by $S_u^{(k)}$ for some $u\in W^0$. 

Given $u\in W$, we say that $u \lessdot_w us_i$ is a cover for the weak order if $\ell(us_i)=\ell(u)+1$. The weak order on $W$ is the transitive closure of these covers. 
We can define operators
\begin{equation}\label{eq:weakop}
\begin{array}{rcl}
\s{i} \colon\  {\mathbb Z}W^0&\longrightarrow& \quad{\mathbb Z}W^0,\\
u\quad&\longmapsto&\ \ \rule{0pt}{28pt} \left\{\begin{array}{ll}
us_i   
&\mbox{ if } u\lessdot_w us_i 
\\      0& \mbox{ otherwise}\end{array}\right.
\end{array} 
\end{equation}
on the weak order of $W$ restricted to $W^0$. The definition and multiplication of $k$-Schur functions is based on the operators $\s{i}$ so it is worthwhile to study
the monoid they generate. As we will see in Example~\ref{ex:kschur} there are difficulties with the behavior of this case which make it very difficult at this point to
understand its combinatorics. For this reason, we will quickly turn our attention to the dual $k$-Schur after Example~\ref{ex:kschur}.

The Pieri rule for $k$-Schur functions is described by certain chains in the weak order of $W$ restricted to $W^0$. This result is given in ~\cite{lapointemorseDEF,Lam,LLMS}. 
A saturated chain $w=\s{i_m}\cdots\s{i_2}\s{i_1}(u)$ in an interval $[u,w]_w$ of the weak order restricted to $W^0$ gives us a sequence of labels $(i_1,i_2,\ldots,i_m)$. We say that the sequence $(i_1,i_2,\ldots,i_m)$ is cyclically increasing if $i_1,i_2,\ldots,i_m$ lies clockwise
on a clock with hours $0,1,\ldots,k$ and if the  $\min\big\{ j : 0\le j\le k; \ j\notin \{i_1,i_2,\ldots,i_m\}\big\}$ lies between $i_m$ and $i_1$.
In particular we must have $1\le m\le k$.
Now, to express the Pieri rule, we first remark that for $1\le m\le k$, the homogeneous symmetric function $h_m $ corresponds to the $k$-Schur function $S_{v(m)}^{(k)}$ where $v(m)$ is a $0$-grassmannian whose main window is given by $|2\;\cdots\;m\;\bar{0}\;m+1\;\cdots\;k\;k+2|$.
Then, the multiplication of a $k$-Schur function $S_u^{(k)}$ by a homogeneous symmetric function $h_m$   
is given by
\begin{equation}\label{eq:PierikSchur}
   S_u^{(k)} h_m:=\ \sum_{(i_1,i_2,\ldots,i_m) \text{ cyclically increasing } \atop \s{i_m}\cdots\s{i_2}\s{i_1}(u)\ne 0}S_{ \s{i_m}\cdots\s{i_2}\s{i_1}(u)}^{(k)}.
\end{equation}
Iterating equation (\ref{eq:PierikSchur}) one can easily see that
\begin{equation}\label{eq:hkschur}
 h_\lambda = \ \sum_{u}  {\rm K}_{\lambda,u}S_{u}^{(k)}
 \end{equation}
is a triangular relation~\cite{lapointemorseDEF}. One way to define $k$-Schur functions is to start with equation~(\ref{eq:PierikSchur}) as a rule, and define them as follows.
\begin{definition}\label{def:kschu}
 The {\sl $k$-Schur functions} are the unique symmetric funtions $S_{u}^{(k)}$ obtained by inverting the matrix $[{\rm K}_{\lambda,u}]$ obtained from (\ref{eq:hkschur}) above.
\end{definition}

 It is clear that we can define a Pieri operator 
   $$H_m=\sum_{(i_1,i_2,\ldots,i_m) \text{ cyclically increasing }}  \s{i_m}\cdots\s{i_2}\s{i_1}\,,$$
   for $1\le m \le k$. Again we can show that $H_aH_b=H_bH_a$ and define $K_{[u,w]_w}$ using the original definition. The example below shows the main  problems we have with this function.

\begin{example} \label{ex:kschur}
Let $k=2$ and $u=|\bar{0}\; 2\; 4|$. We consider the interval $[u,w]_w$ in the weak order  restricted to $W^0$, where $w=|\bar{3}\;4\;5|$. This interval is a single chain $w=\s{0}\s{2}\s{1}(u)$.
In this case, we remark that $$\langle H_1H_1H_1(u),w\rangle=\langle H_1H_2(u),w\rangle=\langle H_2H_1(u),w\rangle=1$$
are the only nonzero entries in $K_{[u,w]_w}$ and we get
$$
 \begin{array} {rcl}
   K_{[u,w]_w} &=& M_{111}+M_{21}+M_{12}\cr
                       &=& F_{12} + F_{21} - F_{111} \cr
                       &=& S_{21} - S_{111}.
 \end{array}$$
\end{example}

\noindent This small example shows some of the behavior of the (quasi)symmetric function $K_{[u,w]_w}$ for the weak order of $W$. In general, it is neither $F$-positive nor Schur positive. Although, these functions contain some information about the structure constants, it is not enough to fully understand them combinatorially. In particular, these functions lack some of the properties needed to use the theory developed in ~\cite{A_JAMS}. The functions $K_{[u,w]_w}$ were first defined in ~\cite{bmws,postnikov} but the combinatorial expansion in terms of Schur functions is still open. 

\subsection{Dual $k$-Schur functions. }

Recall that $Sym = \mathbb Z[h_1,h_2,\dots]$ is the Hopf algebra of symmetric functions. The space of $k$-Schur functions $Sym_{(k)}$ can be seen as a subalgebra of $Sym$ spanned by $ {\mathbb Z}[h_1,h_2,\ldots,h_k]$. In fact, it is a Hopf subalgebra whose comultiplication defined in the homogeneous basis is
given by
 $$\Delta(h_m) =\sum_{i=0}^m h_i\otimes h_{m-i} $$
 and extended algebraically.
 The degree map is given by $\deg(h_m)=m$. 
 The space $Sym$ is a self dual Hopf algebra where the Schur functions $S_\lambda$ form a self dual basis under the pairing $\langle h_\lambda,m_\mu \rangle=\delta_{\lambda,\mu}$. 
 
\noindent The  map dual to the inclusion $Sym_{(k)}\hookrightarrow Sym$, is a projection $Sym\to\!\!\!\!\!\to Sym^{(k)}$, where $Sym^{(k)}=Sym_{(k)}^*$ is the graded dual of $Sym_{(k)}$. It can be checked that the kernel of this projection is the linear span of $\{m_\lambda:\lambda_1>k\}$, hence
  $$Sym^{(k)}\  \cong\  Sym \big/ \langle m_\lambda:\lambda_1>k \rangle\,. $$
The graded dual basis to $S_u^{(k)}$ will be denoted here by ${\frak S}_u^{(k)}=S_u^{(k)*}$ which are also known as the affine Stanley symmetric functions. The multiplication of the dual $k$-Schur ${\frak S}_u^{(k)}$
is described in terms of operator on the affine Bruhat order, as we will see in the next section. 

\subsection{Affine Bruhat order.} 
 For $b-a\leq k$, let $t_{a,b}$ be the transposition in $W$ such that for all $m\in\mathbb Z$, it transposes $a+m(k+1)$ and $b+m(k+1)$.
The \emph{affine Bruhat order}  is given by its covering relation. Namely, for $u\in W$, we  have $u\lessdot ut_{a,b}$ is a cover in the affine Bruhat order if $\ell(ut_{a,b})=\ell(u)+1$.

\begin{proposition}[see \cite{bjornerbrenti}]\label{prop:bruhatcover}
 For $u\in W$ and $b-a\le k$, we have that $u\lessdot ut_{a,b}$ is a cover in the \emph{Bruhat order} if and only if $u(a)<u(b)$ and for all $a<i<b$ we have $u(i)<u(a)$ or $u(i)>u(b)$. 
\end{proposition}

\noindent Notice that if $a'=a+m(k+1)$ and $b'=b+m(k+1)$ then $t_{a',b'}=t_{a,b}$, therefore, many different choices of $a$ and $b$ give the same covering as long as they satisfy the conditions of the proposition. The affine $0$-Bruhat order arises as a suborder of the Bruhat order. We define it by its covers. For $u\in W$, we get  a covering $u\lessdot_0 ut_{a,b}$ if there exists a transposition $t_{a,b}$ satisfying proposition \ref{prop:bruhatcover} and also $u(a)\leq 0<u(b)$. As previously noted, a transposition $t_{a',b'}$ satisfying the same conditions as $t_{a,b}$ gives the same affine Bruhat covering relation as long as $a'\equiv a$, $b'\equiv b$ modulo $k+1$. 
In view of this, we introduce operators on the affine $0$-Bruhat order restricted to $W^0$. To keep track of  the distinct $a,b$ such that $u\lessdot_0 ut_{a,b}$ is an affine $0$-Bruhat covering  for a given $u$.
For any $b-a\le k+1$, let 
\begin{equation}\label{eq:0bruhatop}
\begin{array}{rcl}
\t{\a\b} \colon\  {\mathbb Z}W^0&\longrightarrow& \quad{\mathbb Z}W^0,\\
u\quad&\longmapsto&\ \ \rule{0pt}{28pt} \left\{\begin{array}{ll}
ut_{a,b}   
&\mbox{ if } u\lessdot ut_{a,b} \mbox{ and } u(a)\le 0 < u(b)
\medskip       
\\      0& \mbox{ otherwise.}\end{array}\right.
\end{array} 
\end{equation}
We  write these operators as acting on the right: $u\t{\a\b}$. Remark now that if $u\t{\a\b}\ne 0$, then $u\t{\a\b}=u\t{\a',\b'}\ne 0$ for only finitely many values of $m$ with $a'=a+m(k+1)$ and $b'=b+m(k+1)$. To see this, it is enough to notice that there exists $m$ such that $u(a+m(k+1))\geq 0$  and $m'$ such that  $u(b+m'(k+1))< 0$. 

\begin{example} \label{ex:0Bruhat} Below we have the interval $[|\bar{6}\,8\,3\,\bar{1}\,4\,13|,|8\,\bar{6}\,\bar{2}\,9\,13\,\bar{1}|]$ in the affine $0$-Bruhat graph:
$$
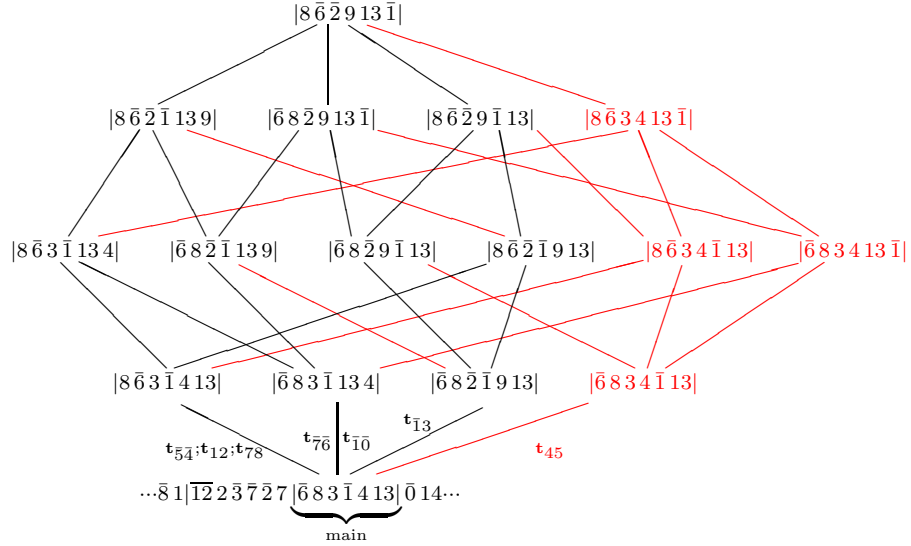
\raise 15pt\hbox{ \begin{picture}(300,190)
      \put(40,-2){$\scriptstyle \cdots \bar{8}\,1|\overline{12}\,2\,\bar{3}\,\bar{7}\,\bar{2}\,7
                              \underbrace{\scriptstyle |\bar{6}\,8\,3\,\bar{1}\,4\,13|}_{\text{main}}\bar{0}\,14\cdots$}
      \put(108,8){\line(-2,1){52}}  \put(50,16){$\scriptscriptstyle \t{\bar5\bar{4}};\t{12};\t{78}$}  
      \put(115,8){\line(0,1){27}}  \put(102,20){$\scriptscriptstyle \t{\bar{7}\bar{6}}\,\,\,\t{\bar{1}\bar{0}}$}  
      \put(120,8){\line(2,1){50}}     \put(140,27){$\scriptscriptstyle \t{\bar{1}3}$}
      \rouge{ \put(130,8){\line(3,1){80}} \put(190,16){$\scriptscriptstyle \t{45}$}}
      \put(30,40){$\scriptstyle |8\,\bar{6}\,3\,\bar{1}\,4\, 13|$}
      \put(90,40){$\scriptstyle |\bar{6}\,8\,3\,\bar{1}\,13\,4|$}
      \put(150,40){$\scriptstyle |\bar{6}\,8\,\bar{2}\,\bar{1}\,9\,13|$}
      \rouge{ \put(210,40){$\scriptstyle |\bar{6}\,8\,3\,4\,\bar{1}\,13|$}}
      \rouge{ \put(148,48){\line(-2,1){78}}  }
      \rouge{\put(52,48){\line(4,1){163}}      }
      \rouge{ \put(204,48){\line(-2,1){78}}  }
      \rouge{ \put(100,48){\line(4,1){158}}     }   
      \put(19,48){\line(-1,1){40}} 
      \put(22,48){\line(3,1){118}}      
      \put(67,48){\line(-2,1){81}}   
      \put(75,48){\line(-1,1){40}}  
      \put(134,48){\line(-1,1){40}} 
      \put(142,48){\line(1,3){13}}  
      \put(-40,90){$\scriptstyle |8\,\bar{6}\,3\,\bar{1}\, 13\,4|$}
      \put(20,90){$\scriptstyle |\bar{6}\,8\,\bar{2}\,\bar{1}\,13\,9|$}
      \put(80,90){$\scriptstyle |\bar{6}\,8\,\bar{2}\,9\,\bar{1}\,13|$}
      \put(140,90){$\scriptstyle |8\,\bar{6}\,\bar{2}\,\bar{1}\,9\, 13|$}
       \rouge{    \put(200,90){$\scriptstyle |8\,\bar{6}\,3\,4\,\bar{1}\, 13|$}}
       \rouge{    \put(250,90){$\scriptstyle |\bar{6}\, 8\,3\,4\, 13\, \bar{1}|$}}
        \rouge{ \put(198,98){\line(-2,5){16}}   }
        \rouge{ \put(-30,98){\line(5,1){200}}     } 
        \rouge{ \put(168,98){\line(-1,1){40}}     } 
        \rouge{ \put(220,98){\line(-4,1){160}}    } 
        \rouge{ \put(100,98){\line(-3,1){120}}      }
       \put(-65,98){\line(2,3){27}}   
       \put(-13,98){\line(-1,2){20}}  
      \put(-8,98){\line(3,4){30}} 
      \put(42,98){\line(-1,5){8}}  
     \put(48,98){\line(1,1){40}} 
      \put(106,98){\line(-1,5){8}}  
      %
      %
      \put(-50,140){$\scriptstyle |8\,\bar{6}\,\bar{2}\,\bar{1}\, 13\,9|$}
      \put(10,140){$\scriptstyle  |\bar{6}\,8\,\bar{2}\,9\,13\,\bar{1}|$}
      \put(70,140){$\scriptstyle |8\,\bar{6}\,\bar{2}\,9\,\bar{1}\,13|$}
     \rouge{ \put(130,140){$\scriptstyle |8\,\bar{6}\,3\,4\, 13\,\bar{1}|$}}
      \put(-31,148){\line(2,1){60}} 
       \put(33,148){\line(0,1){30}} 
       \put(86,148){\line(-3,2){45}} 
   \rouge{   \put(138,148){\line(-3,1){90}}   }
      \put(20,180){$\scriptstyle |8\,\bar{6}\,\bar{2}\,9\,13\,\bar{1} |$}
      \rouge{  \put(154,48){\line(1,3){13}}  }
      \rouge{  \put(153,48){\line(3,2){60}}  }   
      \rouge{  \put(203,98){\line(-3,2){60}} }
      \end{picture}} 
$$
In this example we see that there are three operators from $u=|\bar{6}\,8\,3\,\bar{1}\,4\,13|$ to $w=|8\,\bar{6}\,3\,\bar{1}\, 13\,4|$.
We have $u\t{\bar{5}\bar{4}}=u\t{12}=u\t{78}=w$ labeled by $\bar 4,2,8$, respectively. All other operators  evaluate to 0. For example $u\t{\overline{11}\,\overline{10}}=0$. 
\end{example}

When restricted to $0$-grassmannian permutations, the affine $0$-Bruhat order behaves well, as shown in the next lemma whose proof (for left coset)  can be consulted in~\cite[Prop. 2.6]{LLMS}. Therefore, our operators $\t{ab}$ are well defined.

\begin{lemma}\label{lem:0grass}
If $u\t{ab}=w$ and $u\in W^0$, then we have that $w\in W^0$. 
\end{lemma}

At this point, there are a few questions we would like to answer regarding the monoid ${\mathcal M}\langle \t{ab} \rangle$ generated by the operators $\t{ab}$.
The main questions are: 
\begin{enumerate}
\item[(I)] Can we describe all the relations satisfied by the operators $\t{ab}$ (as in Proposition~\ref{prop:relSchub})?
\item[(II)] Is there a combinatorial object that characterizes all the elements of ${\mathcal M}\langle \t{ab} \rangle$ (as in Corollary~\ref{cor:fixr})? 
\item[(III)] Can we define Pieri operators $H_k$ related to the multiplication ${\mathfrak S}_u h_m$?
\item[(IV)] Can we find a good expression for $H_k$ as in Proposition~\ref{prop:cyclpieri}?
\item[(V)] Is the algebra spanned by the $H_k$ isomorphic $Sym_{(k)}$ (as in Theorem~\ref{thm:youngop})?
\item[(VI)] What is the analogue of Theorem~\ref{THM:Kuw}?
\item[(VII)] Can we show combinatorially the positivity of the structure constants in the product  ${\mathfrak S}_u h_m$  as done in Section~\ref{sec:pos}?
\end{enumerate}
We have some partial answers to question (I) that we will discuss next. Questions (II) and (IV) seem very difficult at this point and are still open.
Questions (III), (V) and (VI) are done in the literature (see~\cite{Lam,benedettib}), although (V) is not stated as it is here. We are in the process of solving question (VII); this involves 
analyzing 3, 4, 5, and 6-tuples of the operators $\t{ab}$ . The number of possibilities are much greater than the situation in Section~\ref{sec:pos}
 and will be available in subsequent work. 

\subsection{Relations of the operators $\t{ab}$.}\label{sec:rel0bruhat}    
The purpose of this section is to understand some of the relations satisfied by the $\t{ab}$ operators restricted to $W^0$. 
Our main goal at this point is not to understand all the defining relations, but to find enough that will allow us to answer question (VII).
Answering question (II) is a very worthwhile project for future work.
Most of the relations we present here were given and proven in~\cite{benedettib}.
The relations  depend  on the following data: for $\t{ab}$ we need to consider $a,b,\overline{a},\overline{b}$ where $\overline{a}$ and $\overline{b}$ are the residue modulo $k+1$ of $a$ and $b$ respectively. Remark that  $\overline{a}\ne \overline{b}$ since $b-a<k+1$. Let $u\in W^0$. Lemma~\ref{lem:0grass} implies  that, if non-zero, $u\t{ab}$ and $u\t{ab}\t{cd}$ are both in $W^0$.  The different relations satisfied by the operators $\t{ab}$ and $\t{cd}$ depend on the relation among $\overline a, \overline b, \overline c, \overline d$. For this reason it is useful to visualize these operators as follows.  

$$
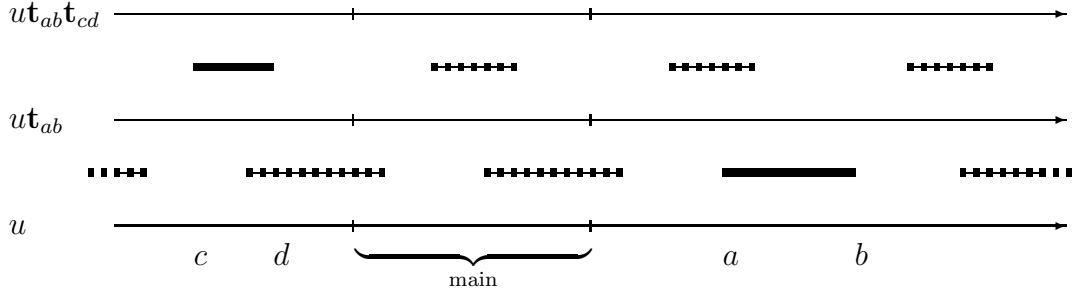
\qquad \raise 10pt\hbox{ \begin{picture}(360,110)
     \put(30,0){$c$} \put(60,0){$d$}  \put(230,0){$a$} \put(280,0){$b$}     
     \put(91,10){$\underbrace{\phantom{|\bar{6}\,8\,3\,\qquad \,\,\,\,\bar{1}\,4\,13|}}_{\text{main}}$}
     \put(-40,12){$u$} \put(0,15){\vector(1,0){360}} \put(90,13){\line(0,1){4}} \put(180,13){\line(0,1){4}}
     \put(-40,52){$u\t{ab}$} \put(0,55){\vector(1,0){360}} \put(90,53){\line(0,1){4}} \put(180,53){\line(0,1){4}}
     \put(-40,92){$u\t{ab}\t{cd}$} \put(0,95){\vector(1,0){360}} \put(90,93){\line(0,1){4}} \put(180,93){\line(0,1){4}}
                      \put(50,35){\line(1,0){50}} \put(140,35){\line(1,0){50}} \put(320,35){\line(1,0){30}} \put(0,35){\line(1,0){10}} 
                      \put(120,75){\line(1,0){30}} \put(210,75){\line(1,0){30}}   \put(300,75){\line(1,0){30}} 
     \linethickness{1mm}
                      \put(230,35){\line(1,0){50}}
                      \multiput(50,35)(5,0){11}{\line(1,0){2}} \multiput(140,35)(5,0){11}{\line(1,0){2}} \multiput(320,35)(5,0){9}{\line(1,0){2}} 
                      \multiput(-10,35)(5,0){5}{\line(1,0){2}}
                      \put(30,75){\line(1,0){30}}
                      \multiput(120,75)(5,0){7}{\line(1,0){2}}  \multiput(210,75)(5,0){7}{\line(1,0){2}}  \multiput(300,75)(5,0){7}{\line(1,0){2}} 
      \end{picture}} 
$$

Above the permutation $u$, the operator $\t{ab}$ is represented by drawing a bold line connecting positions $a,b$ and repeating this pattern to the left and to the right in all positions congruent to $a,b$ modulo $k+1$. Next we apply $\t{cd}$  to the resulting permutation, drawing a bold line connecting positions $c,d$ and repeating that pattern modulo $k+1$. The importance of visualizing not only the bold line but also the dotted ones, relies on the fact that even if in the diagram, the line representing $\t{ab}$ does not intersect the line representing $\t{cd}$. Their ``virtual" copies (or dotted copies) might intersect and this will determine the commutation relation satisfied by these operators. Therefore, it will be important to also consider the pattern produced by these two operators in the main window.

With these definitions in mind we present some of the relations satisfied by the $\mathbf{t}$ operators restricted to $W^0$ (there are less relations if we consider all of $W$).

\smallskip
\noindent {\bf (A)} $\t{ab}\t{cd}\equiv \t{cd}\t{ab}$ \qquad if $\overline{a},\overline{b},\overline{c},\overline{d}$ are distinct. 

\medskip
\noindent {\bf (B1)} $\t{ab}\t{cd}\equiv \t{cd}\t{ab}\equiv 0$ \qquad if ($a<c<b<d$) or ($b=c$ and $d-a>k+1$).

\medskip
\noindent {\bf (B2)} $\t{ab}\t{cd}\equiv 0$ \qquad if ($\overline{a}=\overline{c}$ and $b\le d$) or ($\overline{b}=\overline{d}$ and $c\le a$).

\medskip

\noindent There are more possible zeros than what we present in (B). 
If the numbers $a,b,c,d$ are not distinct, then we must have $b=c$ or $d=a$. If $b=c$, then $d-a\le k+1$ in view of (B). Similarly if $d=a$ then $b-c\le k+1$. 

\medskip
\noindent {\bf (C)} $\t{ab}\t{bd}=\t{ab}\t{b-k-1,a}$ \qquad if $d-a=k+1$,

\medskip
\noindent  if $d-a<k+1$ then there is no relation between $\t{ab}\t{bd}$ and  $\t{bd}\t{ab}$.
Now we look at the cases $\t{ab}\t{cd}$ where $a,b,c,d$ are distinct but some equalities exists between $\overline{a},\overline{b}$ and $\overline{c},\overline{d}$. By symmetry of the relation we will assume that $b<d$, which (excluding (B)) implies that  $a<b<c<d$. 

\medskip
\noindent {\bf (D)} $\t{ab}\t{cd}=\t{d-k-1,c}\t{b-k-1,a}$ \qquad if $\overline{b}=\overline{c}$,\ \  $\overline{d}=\overline{a}$ and $(b-a)+(d-c)=k+1$.

\medskip
All the relations above are {\sl local}. This means that if $\t{ab}\t{cd}=\t{c'd'}\t{a'b'}$, then $|a'-a|$, $|b'-b|$, $|c'-c|$ and $|d'-d|$ are strictly less than $k+1$. For example, in (D) we have $|b-k-1-a|$, $|a-b|$, 
$|d-k-1-c|$ and $|c-d|$ which are strictly less than $k+1$. 

\begin{remark}
The relations we care about in this paper and its sequel are all local. There are some relations that are not local:
$$ \t{ab}\t{cd}=\t{a-k-1,b-k-1}\t{cd}=\t{a+k+1,b+k+1}\t{cd},$$
if $c<a<b<d$. The full description of  the relations of the operators $\t{}$ is rather complicated. It would take too much space here and are not all understood.
\end{remark}

We now consider some more relations of length three:

\medskip
\noindent {\bf (E1)} $\t{bc}\t{cd}\t{ac}\equiv \t{bd}\t{ab}\t{bc}$  \qquad if $a<b<c<d$,

\smallskip
\noindent {\bf (E2)} $\t{ac}\t{cd}\t{bc}\equiv \t{bc}\t{ab}\t{bd}$  \qquad if $a<b<c<d$. 

\medskip\noindent
also we have

\smallskip
\noindent {\bf (F)} $\t{bc}\t{ab}\t{bc}\equiv \t{ab}\t{bc}\t{ab}  \equiv\ {\bf 0}$  \qquad if $a<b<c$ and $c-a<k+1$.  

\begin{remark} \label{rem:addedrel}
\medskip
If we fix a permutation $u$ we can derive more relations of length 2. Let $r=|b-a|+|d-c|$:

\medskip
\noindent {\bf (X1)} $u\t{ab}\t{cd}=u\t{d,c+r}\t{b-r,a}$ \qquad if $r<k+1$, \ $\overline{d}=\overline{a}$, \ $u(c)\le 0$ and $u(d)\le 0$,

\smallskip
\noindent {\bf (X2)} $u\t{ab}\t{cd}=u\t{cd}\t{b-r,b}$ \qquad if $r<k+1$, \ $\overline{d}=\overline{a}$ and $u(d)> 0$,

\smallskip
\noindent {\bf (X3)} $u\t{ab}\t{cd}=u\t{d-r,d}\t{ab}$ \qquad if $r<k+1$, \ $\overline{b}=\overline{c}$ and $u(a+r)\le 0$,

\smallskip
\noindent {\bf (X4)} $u\t{ab}\t{cd}=u\t{d-r,c}\t{b,a+r}$ \qquad if $r<k+1$, \ $\overline{b}=\overline{c}$, \  $u(b)>0$  and $u(a+r)> 0$, 

\smallskip
\noindent {\bf (X5)} $u\t{ab}\t{cd}=u\t{cd}\t{a,b+c-d}$ \qquad if $\overline{b}=\overline{d}$, \  $b-a>d-c$ and $u(d-b+a)> 0$,

\smallskip
\noindent {\bf (X6)} $u\t{ab}\t{cd}=u\t{c,d-b+a}\t{a,b}$ \qquad if $\overline{b}=\overline{d}$, \  $b-a<d-c$ and $u(a)\le 0$. 

\medskip\noindent
In the (X) relations, the conditions we impose on $u$ are minimal to assure that both sides of the equality are non-zero. These conditions are not given by the definition of the operators $\t{ab}$.
For example in (X1), the left hand side is non-zero regardless of the value of $u(d)$ but to guarantee that the right hand side is non-zero, we must have $u(d)\le 0$.
This shows that as operators $\t{ab}\t{cd}\ne \t{d,c+r}\t{b-r,a}$. 
\end{remark}

\subsection{Multiplication of dual $k$-Schur.}
For dual $k$-Schur functions ${\frak S}^{(k)}_u$, the analogue of the Pieri formula~(\ref{eq:PierikSchur}) is given by
 \begin{equation}\label{eq:PieridualkSchur}
   {\frak S}_u^{(k)} h_m:=\ \sum_{u\t{a_1b_1}\cdots\t{a_mb_m}\ne 0 \atop b_1<b_2<\ldots <b_m}{\frak S}_{u\t{a_1b_1}\cdots\t{a_mb_m}}^{(k)},
\end{equation}
where the sum is over all increasing paths $b_1<b_2<\cdots <b_m$ starting at $u$~\cite{LLMS}.
Since the Pieri formula is encoded by increasing composition of operators in the affine $0$-Bruhat order restricted to $W^0$, we can define Pieri operators similar to equation~\eqref{eq:PieriopSchub}
using increasing composition of operators $\t{ab}$. 
We can then define a Pieri operator
 \begin{equation}\label{eq:PieriopkSchub}
 H_m= \sum_{b_1<b_2<\ldots<b_k \atop a_i < b_i} \t{a_1b_1}\t{a_2b_2}\cdots\t{a_mb_m}.
 \end{equation}
Many terms in this sum may be zero. At this point we do not have a good description of the terms that survive or how to express the non-zero terms as in Proposition~\ref{prop:cyclpieri}.
The definition of the operator $H_m$ in this case allows us to see that

By definition of $H_k$ and equation~\eqref{eq:PieridualkSchur}, we have 
\begin{equation}\label{eq:symkH}
 \begin{array}{lr}
  wH_b H_a = \sum_{\nu} d_{w,(a,b)}^u u \hfill &\cr
 \qquad \qquad   \iff  &{\mathfrak S}_w h_a  h_b  =\sum_{u} d_{w,(a,b)}^u {\mathfrak S}_u\,.\cr
 \end{array}
 \end{equation}
In particular, for all $w$ we have  $H_b H_a(w)=H_a H_b(w)$ since $h_ah_b =h_b  h_a $.
\begin{theorem}\label{thm:kBruhatop}
 The algebra  ${\bf B}\langle H_{k}\rangle$ spanned by $\{H_1, H_2, \ldots,H_k\}$ as operators on the $k$-affine Bruhat order restricted to $W^0$ is isomorphic to $Sym_{(k)}$.
\end{theorem}

\proof
As we multiply $H_mH_n$ and $H_nH_m$, some terms  go to zero and others  survive. The terms that survive in  $H_mH_m$  are  of the form 
 $$ \omega=\t{a_1b_1}\t{a_2b_2}\cdots\t{a_mb_m}\t{c_1d_1}\t{c_2c_2}\cdots\t{a_nb_n}.$$
 where $b_1<b_2<\ldots<b_k$ and $d_1<d_2<\ldots<d_n$. Let $d_{(a,b)}^\omega$ be the coefficient of $\omega$ in $H_mH_n$. 
Since $\omega\ne \bf 0$, there is a $u\in W^0$ such that $u\omega=v\ne 0$.
As before, for all $\omega$, we have 
  $$ d_{(a,b)}^\omega = \mbox{Coeff of $v$ in $H_aH_b(u)$} = \mbox{Coeff of $v$ in $H_bH_a(u)$} = d_{(b,a)}^\omega. $$
 Hence $H_aH_b=H_bH_a$.
 
The algebra ${\bf B}\langle H_{k}\rangle$ is clearly spanned by $H_\lambda=H_{\lambda_1}\cdots H_{\lambda_\ell}$ where $\lambda$ runs over all partitions.
Again, we only need to show that the $H_\lambda$'s are linearly independent.
Using the definition of the $H_m$, we have that
${\rm Id}H_\lambda=\sum_\mu d_{\lambda}^{\mu} v_\mu$  where $v_\mu$ is the unique $0$-grassmannian permutation with shape $\mu$ and the $d_{\lambda}^{\mu}$ satisfy
$$h_\lambda(x_1,\ldots,x_r)=\sum_\mu d_{\lambda}^{\mu} s_\mu(x_1,\ldots,x_r).$$
As we have seen in the proof of Theorem~\ref{thm:rBruhatop} this implies the linear independence of the $H_\lambda$.
\qed

As in Section~\ref{sec:K}, let $\langle v,w\rangle=\delta_{v,w}$ define a scalar product on ${\mathbb Z}W^0$. 
For a $u<w$ in the $0$-Bruhat order, we  define the quasisymmetric function
\begin{equation}\label{eq:kSchurKuw}  
    K_{[u,w]} =\sum_\alpha \langle uH_\alpha,w\rangle M_\alpha.
\end{equation}
Again, since $H_aH_b=H_bH_a$, the function $K_{[u,w]} $ is in fact a symmetric function. As shown in \cite{benedettib,bmws} 
\begin{theorem}\label{THM:sSchurKuw}
    $$K_{[u,w]} = \sum_\alpha \langle uR_\alpha,w \rangle F_\alpha = \sum_\mu c_{u,\mu}^w s_\mu\,,$$
    where $ c_{u,\mu}^w$ are defined by
    $$ {\mathfrak S}_u^{(k)} s_\mu = \sum_w   c_{u,\mu}^w  {\mathfrak S}_w^{(k)}\,.$$
Moreover for $\alpha$ a composition of $n$, we have that $\langle uR_\alpha,w \rangle$ count the number of compositions $\omega=\t{a_1b_2}\t{a_2b_2}\cdots\t{a_mb_m}$ such that $u\omega=w$ and 
 $b_i>b_{i+1}$ if and only if $i\in D(\alpha)$.
\end{theorem}

\begin{example} \label{ex:0Bruhat2}
Considering the interval $[u,w]=[|\bar{6}\,8\,3\,\bar{1}\,4\,13|,|8\,\bar{6}\,\bar{2}\,9\,13\,\bar{1}|]$ 
from Example~\ref{ex:0Bruhat}. The total number of composition of operators is  $240$. In this case
  $$K_{[u,w]} =9 F_{1111}+ 30F_{112}+ 51F_{121} + 30F_{13} + 30F_{211}+ 51F_{22} + 30F_{31} + 9F_4 \,,$$
is symmetric and the expansion in term of Schur functions is positive
  $$K_{[u,w]} =9S_4+30S_{31}+21S_{22}+30S_{211}+9S_{1111}\,.$$
\end{example}

\subsection{Comments on the combinatorial proof of the positivity of  $c_{u,\mu}^w$}  

If one considers an interval $[u,w]$ of rank 3 
and computes $K_{[u,w]}$, then by Theorem~\ref{eq:kSchurKuw} the coefficient of $F_{21}$ and $F_{12}$ must be the same in $K_{[u,w]}$. This means that every time we have a descent
followed by an ascent in a chain, we must have another chain with an ascent followed by a descent. This should be reflected in relations like (X) and could depend on $u$.
To acheive a result similar to~\cite{ABS} for $K_{[u,w]}$, one needs to  first to build a full set of relations of length 3 that pairs every ascent-descent type to a descent-ascent. This cannot be done independently from $u$.
The purpose of this will be to define Dual-Knuth operations on the maximal chains in intervals $[u,w]$ in order to construct dual graphs as in~\cite{A_JAMS}.

We give here a partial list of the relations of length 3 that would be the analogue  for dual $k$-Schur of (A)--(B)--(C) in Section~\ref{sec:pos}.
The complete full list of 3-relations needed is too long for this survey. In future work, we will need to show that the corresponding $\phi_i$ defined by those relations satisfy the axioms of \cite{A_JAMS}.
This is a long analysis that will appear in subsequent work. This will show that the monoid defined by the $\t{ab}$ behaves like the monoid of Section~\ref{sc:New}, even if it does not satisfy
the Fomin and Greene's hypothesis. This shows that these monoids are worthwhile to investigate.

We have already listed some of the relations satisfied by triplets of operators $\t{ab}$. Relations \textbf{(A),(E1),(E2),(F)} resemble the relations listed in (\ref {eq:relschub}). However, as noted before in the case of the operators $\t{ab}$, more relations can be derived making them more complex that the $\u{ab}$ operators.

\begin{equation*}
\begin{array}{clrclll}
(1a)&&\t{ab}\t{cd}\t{ec}&\equiv&\t{ec}\t{a,b-|c-e|}\t{ed},\hfill&&
        \hbox{if $\a<b<e<c<\d$ and $\bar a=\bar d<\bar e<\bar b=\bar c$}\hfill\\
(1b)&&\t{ab}\t{cd}\t{ec}&\equiv&\t{\bar d c}\t{\bar b a}\t{ec},\hfill&&
        \hbox{if $\a<b<e<c<\d$ and $\bar a=\bar d=\bar e,\, \bar b=\bar c$}\hfill\\        
(1c)&&\t{ab}\t{cd}\t{ef}&\equiv&\t{ef}\t{ab}\t{cd},\hfill&&
        \hbox{if $\a<b<c<e<f<\d$ and $\bar a=\bar d,\, \bar b=\bar c$}\hfill\\         
 (1d)&&\t{ab}\t{bc}\t{db}&\equiv&\t{db}\t{ad}\t{dc},\hfill&&
        \hbox{if $\a<d<b<c$ and $\bar a=\bar c$}\hfill\\     
(1e)&&\t{ab}\t{bc}\t{db}&\equiv&\t{ab}\t{b-m,c-m}\t{db},\hfill&&
        \hbox{if $\a=d<b<c$ and $\bar a=\bar c,\, m=k+1$}\hfill\\     
              
 \end{array}
 \end{equation*}       
 
In analogy with relations \textbf{(X1)-(X6)}, let us list more relations that depend on the permutation $u$ we apply them to. Let $r=|d-c|+|b-a|<k+1$       
  \begin{equation*}
  \begin{array}{clrclll}      
(2a)&u\t{ab}\t{cd}\t{ef}\equiv&u\t{d,c+s}\t{b-s,a}\t{ef},\hfill&
        \hbox{if $\a<b<e<f\leq c<\d$, $\bar a=\bar d,\, u(c)\leq 0,\,u(d)\leq 0$}\hfill\\
(2b)&u\t{ab}\t{cd}\t{ef}\equiv&u\t{cd}\t{b-r,b}\t{ef},\hfill&
        \hbox{if $\a<b<e<f\leq c<\d$, $\bar a=\bar d,\, u(d)>0$}\hfill\\

(3a)&u\t{ab}\t{cd}\t{ef}\equiv&u\t{d-r,c}\t{b,a+r}\t{ef},\hfill&
        \hbox{if $\a<b<e<f\leq c<\d$,$\bar b=\bar c$, $\bar e\geq \bar d,\,u(a+r)>u(b)>0$}\hfill\\
(3b)&u\t{ab}\t{cd}\t{ef}\equiv&u\t{d-r,d}\t{ab}\t{ef},\hfill&
        \hbox{if $\a<b<e<f\leq c<\d$, $\bar b=\bar c$,\, $\bar e\geq \bar d,\,u(a+r)\leq0$}\hfill\\
(4a)&u\t{ab}\t{cd}\t{eb}\equiv&u\t{eb}\t{ae}\t{c-|b-e|,d},\hfill&
        \hbox{if $\a<e<b<c<\d$, $\bar b=\bar c,\bar e>\bar a,\, u(a+r)>u(b)>0$}\hfill\\     
\end{array}
\end{equation*}
\vskip -15pt 
$$ \,\,(4b) \quad u\t{eb}\t{ae}\t{c-|b-e|,d}\equiv u\t{eb}\t{d-r,d}\t{ab},\,
        \hbox{if $\a<e<b<c<\d$, $\bar b=\bar c,\bar e>\bar a,\, u(a+r)\leq0$}$$          

If the reader represents these relations as a system of bars, they can be interpreted as exchanging an ascent-descent by a descent-ascent. As an example, putting $b'=b-|c-e|$ in relation $(1a)$ we can represent it graphically as 

$$
\setlength{\unitlength}{.5cm}
\begin{picture}(12,5)
\thicklines
\put(1,2){\line(-2,0){2}}
\put(5,3){\line(-2,0){2}}
\put(3.1,4){\line(-2,0){1.5}}
\put(-1.1,1.5){$a$}\put(0.6,1.4){$b$}
\put(3,2.5){$c$}\put(4.5,2.4){$d$}
\put(1.6,3.5){$e$}\put(2.9,3.5){$c$}
\put(5.6,2.8){$\equiv$}
\put(8.5,3){\line(-2,0){2}}
\put(11.7,2){\line(-2,0){2}}
\put(12.4,4){\line(-2,0){2.8}}
\put(6.4,2.4){$a$}\put(8.1,2.3){$b'$}
\put(9.7,1.5){$e$}\put(11.2,1.5){$c$}
\put(9.7,3.5){$e$}\put(11.9,3.4){$d$}
\end{picture}$$

Next we list more ascent-descent relations equivalent to descent-ascent. This is not an exhaustive list but it gives a good sense of the behaviour of these operators.

  \begin{equation*}
  \begin{array}{clrclll}      
(6a)&&u\t{ea}\t{ab}\t{cd}&\equiv&u\t{eb}\t{c,d-|a-e|}\t{ea},\hfill&&
        \hbox{if $c<d<e<a<b$, $\bar c<\bar e,\,\bar a=\bar d,\, u(b-r)\leq 0$}\hfill\\
(6b)&&u\t{ea}\t{ab}\t{cd}&\equiv&u\t{eb}\t{c,c+r}\t{ab},\hfill&&
        \hbox{if $c<d<e<a<b$, $\bar c\leq e,\,\bar a=\bar d,\, u(b-r)>0$}\hfill\\        
(6c)&&u\t{ea}\t{ab}\t{cd}&\equiv&u\t{eb}\t{d,c+r}\t{b-r,a},\hfill&&
        \hbox{if $c<d<e<a<b$, $\bar c>\bar e,\,\bar a=\bar d,\, u(b-r)\leq 0$}\hfill\\
(6d)&&u\t{ea}\t{ab}\t{cd}&\equiv&u\t{eb}\t{d-r,c}\t{b,a+r},\hfill&&
        \hbox{if $c<d<e<a<b$, $\bar c\neq \bar e\leq \bar d,\,\bar b=\bar c,\, u(c)>0$}\hfill\\        
(6e)&&u\t{ea}\t{ab}\t{cd}&\equiv&u\t{eb}\t{cd}\t{a,a+r},\hfill&&
        \hbox{if $c<d<e<a<b$, $\bar c\neq \bar e\leq \bar d,\,\bar b=\bar c,\, u(c)\leq 0$}\hfill\\        
\end{array}
\end{equation*}

We encourage the reader to draw the corresponding diagrams of the given relations together with their \emph{virtual} copies in order to realize what these relations look like and understand better the interaction of these triplets. A full understanding of the relations satisfied by tuples of the operators $\t{ab}$ will lead us to describe connected components of these relations. This is work in progress that we aim to use, for instance, to solve question (VII) as stated before.

\begin{remark}
In a recent paper, Assaf and Billey  \cite{AB} have constructed involutions $\phi_i$ on the so called star-tableaux. Such involutions preserve the spin statistic. 
Star-tableaux are equivalent to non-zero sequences of operators $\t{ab}$
acting on the identity 0-grassmannian permutation $Id$. These transformations $\phi_i$ are strongly related to the relations we study satisfied by triplets $\t{ab}$. Showing that these triplets also satisfy the spin statistic will in fact give us a much stronger positive result.
We expect to include this as well in future work.
\end{remark}


\begin{thebibliography}{9}

  

\bibitem[ABS]{ABS}
M. Aguiar, N. Bergeron and F. Sottile,
\emph{Combinatorial Hopf Algebra and generalized Dehn-Sommerville relations\/},
Compositio Mathematica {\bf 142-1} (2006)  pp 1--30. 

\bibitem[Assaf]{A_JAMS}
S. Assaf, {\it combinatorial proof of LLT and Macdonald positivity}.
submitted (arXiv:1005.3759).

\bibitem[AB]{AB}
S. Assaf and Sara Billey
\emph{Affine dual equivalence and k-Schur functions}. (arXiv:1201.2128).

\bibitem[ABF]{assafbsottile}
S. Assaf, N. Bergeron and F. Sottile,
\emph{On a Positive Combinatorial Construction of  Schubert Coefficients: Schubert vs Grassmanian} [in preparation] (2012).

\bibitem[Beligan]{beligan}
M. Beligan {\it Insertion for Tableaux of Transpositions, A Generalization of SchenstedÕs Algorithm}, Ph. D. Thesis, York University (2007) 109p.

\bibitem[BB12]{benedettib}
C. Benedetti and N. Bergeron,
\emph{Schubert Polynomials and $k$-Schur functions} submitted (2012) (arXive 1209.4956 ).

\bibitem[BBTZ]{BBTZ}
C. Berg, N. Bergeron, H. Thomas and M. Zabrocki, 
\emph{Expansion of k-Schur functions for maximal k-rectangles within the affine nilCoxeter algebra}, arXiv:1107.3610.

\bibitem[BLL]{BLL}
N. Bergeron, T. Lam, and H. Li,
\emph{Combinatorial Hopf algebras and towers of algebras -- dimension, quantization and functorality}, 
Algebr. Represent. Theory {\bf 15-4} (2012)  675--696.

\bibitem[BMSW]{bmws}
N. Bergeron, S. Mykytiuk, F. Sottile, and S. van Willigenburg,
\emph{Pieri Operations on Posets}, 
J. of Comb. Theory, Series A {\bf 91} (2000) 84--110.

\bibitem[BS98]{bsottileschub}
N. Bergeron and {F. Sottile},
\emph{Schubert polynomials, the Bruhat order, and the geometry of flag manifolds},
 Duke Math. J.  {\bf 95-2} (1998)  373--423.
%

\bibitem[BS99Mono]{bsottilemonoid}
N. Bergeron and F. Sottile,
\emph{A monoid for the Grassmannian-Bruhat order},
Europ. J. Combinatorics {\bf 20} (1999) 197--211.

\bibitem[BS02]{bsottileskew}
N. Bergeron and F. Sottile,
\emph{Skew Schubert Functions and the Pieri Formula for the Flag Manifolds},
Trans. Amer. Math. Soc. {\bf 354} (2002) 651--673.

\bibitem[BJN]{BJN}
R. Biagioli, F. Jouhet and P. Nadeau, \emph{Fully commutative elements and lattice walks},
FPSAC 2013, DMTCS proc.

\bibitem[BB05]{bjornerbrenti}
A. Bjorner and F. Brenti,
\emph{Combinatorics of Coxeter groups}, Graduate Texts in Mathematics {\bf 231}, Springer, New
York, 2005. 
%
%

\bibitem[F95]{F95}
S. Fomin, 
{\it Schur operators and Knuth correspondences},
J. Combin. Theory Ser. A {\bf 72-2} (1995) 277--292. 

\bibitem[FG98]{FG98}
S. Fomin and C. Greene {\it Noncommutative Schur functions and their applications}. Selected papers in honor of Adriano Garsia (Taormina, 1994). Discrete Math.  {\bf 193-1-3}  (1998) 179--200.

\bibitem[Hum]{humphreys}
J. E. Humphreys {\it Reflection groups and Coxeter groups}. Cambridge University Press.  (1992).


\bibitem[KLS]{positroid} 
A. Knutson, T. Lam and David Speyer,
\emph{Positroid Varieties: Juggling and Geometry},
arxiv: 1111.3660 (2011).
%

\bibitem[Lam]{Lam}
T. Lam, 
\emph{Schubert polynomials for the affine Grassmannian},
J. Amer. Math. Soc. {\bf 21-1} (2008) 259--281.

\bibitem[LLMS]{LLMS}
T. Lam, L. Lapointe, J. Morse and M. Shimozono, 
\emph{Affine insertion and Pieri rules for the affine Grassmannian},
Mem. Amer. Math. Soc. {\bf 208-977} (2010).

\bibitem[LLM]{LLM03}
L. Lapointe, A. Lascoux, and J. Morse, 
\emph{Tableau atoms and a new macdonald positivity conjecture}, Duke Math. J. {\bf 116}
(2003) 103--146.

\bibitem[LM05]{LM1} L.~Lapointe and J.~Morse, 
\emph{Tableaux on $k+1$-cores, reduced words for affine
permutations, and $k$-Schur expansions}, J. Combin. Theory Ser. A
\textbf{112-1} (2005)  44--81.


\bibitem[LM07]{lapointemorseDEF}
L. Lapointe and J. Morse, \emph{ A k-tableaux characterization of k-Schur functions}, Adv. Math.
{\bf 213-1} (2007) 183--204.
%

\bibitem[LS81]{LS81}
A.~Lascoux and M.-P. Sch{\"u}tzenberger, {\em Le monoide plaxique}, Quaderni de la ricerca scientifica {\bf 109} (1981) 129--159.

\bibitem[LS82]{LS82a}
A.~Lascoux and M.-P. Sch{\"u}tzenberger, {\em Polyn{\^o}mes de
  {S}chubert}, C. R. Acad. Sci. Paris, {\bf 294} (1982) 447--450.

\bibitem[M91]{Macdonald91}
I.~G. Macdonald, {\em Notes on {S}chubert Polynomials}, Laboratoire de
  combinatoire et d'informatique math\'ematique {(LACIM)}, Universit\'e du
  Qu\'ebec \`a Montr\'eal, Montr\'eal, 1991.

\bibitem[M95]{macdonald} {I.G. Macdonald}, emph{Symmetric functions and Hall polynomials}, 2nd ed., Oxford University Press, 1995.

\bibitem[Post]{postnikov}
A. Postnikov, \emph{Affine approach to quantum Schubert calculus},
Duke Math. J. \textbf{128-3} (2005) 473--509. 

\bibitem[Sagan]{sagan} B. Sagan, \emph{The Symmetric Group: Representations, Combinatorial Algorithms, and Symmetric Functions}, Graduate Texts in Mathematics, 203 (2000).

\bibitem[Sot]{Sottile}
F. Sottile, \emph{Pieri's formula for flag manifolds and Schubert polynomials},
Ann. de l'institut Fourier \textbf{46-1} (1996) 89--110.

\bibitem[Stan84]{stan84} R. Stanley, \emph{On the number of reduced decompositions of elements of Coxeter groups}, Europ. J.
Combin., {\bf 5} (1984) 359--372.

\bibitem[Stan00]{stanley} R. Stanley, \emph{Positivity problems and conjectures in algebraic combinatorics},
Mathematics: frontiers and perspectives, Amer. Math. Soc., Providence, RI, (2000)  295--319.

\bibitem[Stemb]{Stemb} 
J. R. Stembridge, \emph{The enumeration of fully commutative elements of Coxeter groups},
J. Algebraic Combin.
{bf 7-3} (1998) 291--320.

   
  \end{thebibliography}
\end{document}